\documentclass[12pt]{psp}
 \usepackage{cmap}
 \usepackage{amsmath,amsxtra,amscd,amssymb,latexsym,stmaryrd, mathtools,mathrsfs,xfrac}
  \usepackage{xcolor}
  \usepackage{bbm,enumitem}

\usepackage{lmodern} 
\usepackage[utf8]{inputenc}
\usepackage[T1]{fontenc}

\usepackage[colorlinks=true,linkcolor=red!70!black,citecolor=green!70!black, urlcolor=magenta!70!black,backref]{hyperref}

\newtheorem{theo}{Theorem}[section]
\newtheorem{coro}[theo]{Corollary}
\newtheorem{prop}[theo]{Proposition}
\newtheorem{lemm}[theo]{Lemma}

\newtheorem{defi}[theo]{Definition}

\newnumbered{rema}[theo]{Remark}
\newnumbered{exem}[theo]{Example}
\newnumbered{ques}[theo]{Question}
\newnumbered{prob}[theo]{Problem}

\newcommand{\proba}{\operatorname{\mathcal{P}}}
\newcommand{\esp}{\operatorname{\mathbb{E}}}
\newcommand{\Lip}{\operatorname{Lip}}

\newcommand{\wass}{\operatorname{W}}
\newcommand{\cost}{\operatorname{C}}
\newcommand{\rv}[1]{\textsc{\MakeLowercase{#1}}}
\DeclareMathOperator{\grad}{\nabla}

\DeclareMathOperator{\diam}{diam}
\DeclareMathOperator{\Hol}{Hol}
\renewcommand{\div}{\operatorname{\nabla\cdot}}
\newcommand{\Id}{\operatorname{Id}}
\newcommand{\op}[1]{\operatorname{#1}}

\newcommand{\Cku}[1]{\operatorname{\mathcal{C}^{#1}_1}}

\newcommand*{\dd}%
  {\relax\ifnum\lastnodetype>0\mskip\medmuskip\fi\mathrm{d}}

\title[Optimal transportation for IFS]{Optimal transportation and stationary measures for iterated function systems}

\author[Beno\^{\i}t R. Kloeckner]{BENO\^IT R. KLOECKNER \\ Univ Paris Est Creteil, CNRS, LAMA, F-94010 Creteil, France \addressbreak
Univ Gustave Eiffel, LAMA, F-77447 Marne-la-Vallée, France\addressbreak
e-mail\textup{: \texttt{benoit.kloeckner@u-pec.fr}}}

\begin{document}

\maketitle

\begin{abstract}
In this article we show how ideas, methods and results from optimal transportation can be used to study various aspects of the stationary measures of Iterated Function Systems equipped with a probability distribution. We recover a classical existence and uniqueness result under a contraction-on-average assumption, prove generalized moment bounds from which tail estimates can be deduced, consider the convergence of the empirical measure of an associated Markov chain, and prove in many cases the Lipschitz continuity of the stationary measure when the system is perturbed, with as a consequence a ``linear response formula'' at almost every parameter of the perturbation.
\end{abstract}


\setcounter{tocdepth}{1}
\tableofcontents

\section{Introduction}\label{s:sample}

Let $(X,d)$ be a complete separable metric space (endowed with its Borel $\sigma$-algebra for all measurability purposes) and $\Phi = \{\phi_i \colon i\in I\}$ be an \emph{Iterated Functions System} (IFS), i.e. a family of continuous maps $\phi_i : X\to X$ indexed by a set $I$, either countable or endowed with a standard $\sigma$-algebra. The set of probability measures on $X$ is denoted by $\proba(X)$.

Hutchinson \cite{hutchinson1981fractals} introduced such IFS to produce fractals: under a contraction hypothesis, there is a unique compact subset $K_\Phi$ of $X$ such that
\[K_\Phi = \bigcup_{i\in I} \phi_i(K_\Phi).\]
The proof is very simple: one shows that the map $K\mapsto \cup_i \phi_i(K)$ is a contraction in the Hausdorff metric, and applies the Banach fixed point theorem.
Given additionally $\eta\in \proba(I)$, one is interested in  existence, uniqueness and properties of a measure $\mu\in\proba(X)$ such that
\begin{equation}
\mu = \int (\phi_i)_*\mu \dd \eta(i)
\label{e:stationary}
\end{equation}
i.e. $\int f(x) \dd\mu(x) = \iint f\circ \phi_i(x) \dd \mu(x) \dd \eta(i)$ for all $f\in C_b(X)$, the set of bounded continuous functions $X\to\mathbb{R}$.
Such a measure is called a \emph{stationary} measure for the pair 
$(\Phi,\eta)$, which is sometimes called an Iterated Functions System \emph{with probabilities} but that we will call an IFS for simplicity, since we will only consider this case. When the $\phi_i$ are contractions, again existence and uniqueness mostly follow from the Banach fixed point theorem; Hutchinson used the now-called Wasserstein distance of exponent $1$ (in its dual formulation, restricted to compactly supported measures). As we shall see, using general Wasserstein distances one can use the fixed point theorem approach and get moment estimates at the same time.

We shall also be concerned with the attractivity of the stationary measure $\mu$. Assume $(\rv{i}_k)_{k\in\mathbb{N}}$ are independent, identically distributed random variables of law $\eta$ and $\rv{x}_0$ is a random variable independent from them. Then one constructs a Markov chain, named the ``chaos game'' by Barnsley \cite{barnsley1988book}, by setting $\rv{x}_{k+1} = \phi_{\rv{I}_k}(\rv{x}_k)$. Two sequences of measures are then to be studied: the laws of the $\rv{x}_k$, and the ``empirical measures'' $\frac1k\sum_{j=1}^k\delta_{\rv{x}_j}$.

The goal of the present article is to apply tools and ideas from optimal transportation in this context, to show the variety of information they provide with simple (while not always elementary) proofs; we shall also get inspiration from the ``thermodynamical formalism''; more precisely we use the transfer operator (also known as the Markov operator in the present context) in a crucial way. Results are stated and proved throughout the article, but to give a motivation we give at the end of this introduction the most striking application, in the well-known example of Bernoulli convolutions.

\subsection{Brief comparison of the literature with the present results}

The stationary measures of IFS have been investigated in a huge number of articles, and this overview is necessary partial.
In \cite{barnsley1988newclass}, one of the early works popularizing the field, the existence and uniqueness of the stationary measure was proven under a hypothesis of contraction on average, and moment estimates where given. As a warm-up, we will reprove their main results using the point of view of optimal transport in Section \ref{s:e-u-m}. 
In Section \ref{s:moments}, using the transfer operator we obtain new generalized moment estimate, see e.g. Corollary \ref{c:HPmoment}.

Ergodicity, i.e. weak convergence of the empirical measures toward the stationary measure, is an important property as it enables a Monte Carlo Markov Chain approach to approximate the stationary measure. Ergodicity was proved in \cite{elton1987ergodic,elton1990ergodic}, see also \cite{forte1998ergodic, szarek2003polish}. In Section \ref{s:empirical}, we give almost optimal bounds for convergence of the empirical measure in the $L^1$ Wasserstein distance and other related metrics, thus providing a ``quantitative ergodicity'' statement (Theorem \ref{th:Markov}). This result is new, but follows immediately from a general result on Markov chains obtained previously in \cite{kloeckner2018empirical}; it is mentioned as an illustration of the wealth of tools available.

Tail estimates for the stationary measure have been largely investigated, notably for affine IFS in \cite{kesten1973random,goldie1991implicit} and more recently \cite{kevei2016noteKGG}. These works use assumptions that induce a polynomial tail; in Section \ref{s:moments} we will be interested in quite general and easy to obtain, but less precise, tail estimates. We will in particular consider a case with exponential tail, which (obviously) does not fit Kesten's assumptions (Corollary \ref{c:mainHP}).

In the case of two affine transformations of the line, one contracting and one expanding, \cite{bergelson2006affine} studied more general invariant measures, where the random choice of indices need not be independent but is defined by a shift-invariant measure on $I^{\mathbb{N}}$. In Section \ref{s:beyond-products}, we use a modified Wasserstein distance to generalize further this setting to skew products; under a uniform contraction assumption, we get explicit convergence rates towards the stationary measure. For such generalizations, ergodicity has been studied in \cite{elton1990ergodic,silvestrov1998ergodic}.

In addition to the aforementioned \cite{madras2010quantitative}, let us cite two other  works studying IFS and Wasserstein distances. First \cite{galatolo2016approximation} focusses on the computation of the stationary measure, with certified bounds on the Wasserstein distance from the approximation. 
Second, \cite{fraser2015moments} computes the exact value of the Wasserstein distance between the stationary measures of different affine IFS on the line. In section \ref{s:response} we study how much the stationary measure depends on the underlying IFS; this is less precise but more general than the previous cited article. Our interest in such bounds is that they enable us to use a sophisticated Rademacher theorem for measures to obtain a linear response formula (see Corollary \ref{c:linearBernoulli}). While linear response has become a classical subject in dynamical system, to my knowledge this result is the first of this kind in the context of IFS; it has one weakness and one strength: it is only proven for \emph{almost} all parameters, but we get a derivative in quite a strong sense.

Let us conclude with some directions we do not pursue in this work. A prominent topic in the study of IFS on $\mathbb{R}^d$ is to determine the dimension of their stationary measure and whether it is absolutely continuous with respect to Lebesgue measure ---this includes open questions even for some of the simplest IFS. Citing all relevant works would be daunting, and while we will use Solomyak's theorem \cite{solomyak1995erdos}, we will not be primarily concerned with such questions here. We thus refer the reader to the recent survey \cite{varju2018recent} and references therein.

More general conditions ensuring existence, uniqueness and attractivity of the stationary measure have been sought, a notable example being the local contraction condition introduced by \cite{steinsaltz1999locally}, where the focus is on the behavior of the backward iteration. For IFS satisfying this local contraction assumption, exponential convergence of the law in the Wasserstein distance to the stationary measure was proven in \cite{madras2010quantitative}. In a very general setting, without the contraction assumption, \cite{barnsley2011chaos} manages to prove that the attractor of an IFS is obtained as the limit of a random orbit. The case of place-dependent probabilities is for example considered in \cite{barnsley1988placeDependent,szarek2003polish}.

Inspired by the thermodynamical formalism for expanding  dynamical systems, statistical properties (e.g. Central Limit Theorem, Invariance Principle) where studied in \cite{peigne1993IFSspectral, pollicott2001transfer,walkden2007invariancePrinciple,santos2013limitLaws}. Our strong reliance on transfer operators is quite similar in spirit, with the same inspiration from expanding dynamical systems, but the present paper focuses on different results.


\subsection{Linear response for Bernoulli convolutions}\label{subs:Bernoulli}

The \emph{Bernoulli convolution} $\mu_\beta$ (where $\beta\in(0,1)$) is defined as the stationary measure of the following classical IFS $(\Phi^\lambda,\eta)$: 
\begin{align*}
I &=\{0,1\}, &\eta(\{0\})&=\eta(\{1\}) = \frac12\\
\phi_0^\lambda(x) &= \lambda x, & \phi_1^\lambda(x) &= \lambda x +(1-\lambda) &\forall x\in\mathbb{R}.
\end{align*}
(The precise value $1-\lambda$ of the translation part in $\phi_1^\lambda$ has no particular relevance ---as soon as it is not zero--- with this value the attractor is $[0,1]$ for all $\lambda>\frac12$, but it bears no consequence on the result below.)

We shall prove in Section \ref{s:response} that the map $\lambda\mapsto \mu_\lambda$ is Lipschitz in the Wasserstein distances $\wass_q$ of all exponents $q$; for $q>1$, thanks to the differentiation theorem of Ambrosio, Gigli and Savar\'e \cite{ambrosio2008gradient} this implies an almost-everywhere \emph{linear response formula} (a terminology coined in dynamical systems, see e.g. \cite{ruelle1998linear,ruelle2009review,baladi-smania2012linear}): the map $\lambda\mapsto \mu_\lambda$ can be differentiated in some precise sense at almost-all $\lambda$, and while we do not get an explicit expression for the differential we show that it takes a specific form. We state here the result with $q=2$, see Corollary \ref{c:LinearResponse} for  a more general result.
Let $\Id$ denote the identity map of $\mathbb{R}$, so that for a function $w$ and number $x,\varepsilon$ we have $(\Id+\varepsilon w)(x)=x +\varepsilon w(x)$. Let $\varphi_*\mu$ denote the push-forward of a measure $\mu$ by a map $\varphi$.
\begin{coro}\label{c:linearBernoulli}
The family of Bernoulli convolutions $(\mu_\lambda)_{\lambda\in(0,1)}$ is differentiable almost everywhere in the quadratic Wasserstein space, meaning that there exist a family $(v_\lambda)_{\lambda\in(0,1)}$ of $L^2(\mu_\lambda)$ functions such that for Lebesgue-almost all $\lambda\in(\frac12,1)$:
\begin{equation} \wass_2\big(\mu_{\lambda+\varepsilon} , (\Id+\varepsilon v_\lambda)_*\mu_\lambda \big) = o(\varepsilon) \qquad\text{as }\varepsilon\to 0 
\label{e:diff}
\end{equation}
As a consequence, there exist a family $(w_\lambda)_{\lambda\in(\frac12,1)}$ of measurable functions and a Lebesgue-negligible set $E$ such that for all $\lambda\in(\frac12,1)\setminus E$ and for all smooth compactly supported $f:\mathbb{R}\to \mathbb{R}$, 
\[ \frac{\dd}{\dd t}\Big\rvert_{t=\lambda} \int f \dd\mu_t  = \int_0^1 f'(x) w_\lambda(x) \dd x.\]
Moreover,  for Lebesgue-almost all $\lambda>1/\sqrt{2}$ we have $w_\lambda\in L^q([0,1])$ for all $q>1$.
\end{coro}

\section{Notation and definition of Wasserstein distances}\label{s:Wasserstein}

Let us now introduce briefly the Wasserstein distances originating in optimal transportation theory. We only mention here statements that will be used several times or are relevant to several parts of the text. On several occasions below we will use results from the literature in a crucial way without giving their full statements; we shall only do so when we can give a precise reference, use the result as it is stated without modification, and when restating it would be somewhat redundant with the corollary we get from it. This makes the present article not as self-contained as it could be, but is consistent with the purpose of showing what optimal transportation can bring to the subject and encourage the reader to learn more about it. For details and proofs of the claims made in this section, see for example \cite{Villani2009OldNew}.

Let us fix a reference point $x_0\in X$. This choice can be arbitrary and has no conceptual bearing, but can be subject to optimization in some cases.
For each $q\in(0,+\infty)$, the \emph{$q$-th moment} of $\mu\in\proba(X)$ is
\[m_{x_0}^q(\mu) := \int d(x,x_0)^q \dd\mu(x) \in[0,+\infty].\]
The set of probability measures $\mu$ of finite $q$-th moment is denoted by $\proba_q(X)$ and does not depend on $x_0$.

Given measures $\mu_0,\mu_1\in\proba(X)$, the set of \emph{transport plans} or \emph{couplings} is the set $\Gamma(\mu_0,\mu_1)$ of measures $\gamma\in\proba(X\times X)$ such that $\gamma(A\times X) = \mu_0(A)$ and $\gamma(X\times B)=\mu_1(B)$ for all measurable $A,B\subset X$. The number $\gamma(A\times B)$ can be interpreted as the amount of mass moved from $A$ to $B$ under the plan $\gamma$.

One defines the \emph{total cost} and \emph{Wasserstein distance} of exponent $q$ between two probability measures by:
\begin{align*}
\cost_q(\nu_0,\nu_1) &= \inf_{\gamma\in\Gamma(\nu_0,\nu_1)} \int d(x,y)^q \dd\gamma(x,y)\\
\wass_q(\nu_0,\nu_1) &= \cost_q(\nu_0,\nu_1)^{\min(1,\frac1q)}.
\end{align*}
Observe that the Wassertein distance of exponent $q<1$ is actually the Wassertein distance of exponent $1$ of $(X,d^q)$. The cost and the Wasserstein distance are finite as soon as $\mu_0,\mu_1\in\proba_q(X)$, and $(\proba_q(X),\wass_q)$ is a complete metric space. Convergence in the Wasserstein distance is stronger than weak-$*$ convergence when $X$ is not compact; if $X$ is compact, then $\wass_q$ metrizes the weak-$*$ topology.

Note that when $\mu_0=\delta_x$ is a Dirac mass, there is only one possible coupling: $\Gamma(\delta_x,\mu_1) = \{\delta_x\otimes \mu_1\}$; therefore the $q$-th moment of $\mu$ can be expressed as $m_{x_0}^q(\mu) = \cost_q(\delta_{x_0},\mu)$.

We denote by $\Lip(\phi)$ the Lipschitz constant of a map $\phi:X\to\mathbb{R}$, i.e.
\[\Lip(\phi) = \inf\big\{C\ge 0  \colon \lvert \phi(x)-\phi(y) \rvert \le Cd(x,y) \ \forall x,y\in X\big\} \in [0,+\infty].\]
The \emph{Kantorovich duality} expresses that $\wass_1$ coincides with the value of a ``dual'' optimization problem:
\[\wass_1(\mu_0,\mu_1) = \sup_{\Lip(f)\le 1} \big\lvert \int f\dd\mu_0 - \int f \dd\mu_1 \big\rvert.\]
Similarly, for all $q<1$ the Wasserstein distance $\wass_q$ can be expressed as the maximal differences between the integrals of $q$-H\"older functions with H\"older constant $1$, since $q$-H\"older functions are precisely the Lipschitz function of the metric $d^q$.

\section{Wasserstein contraction and its consequences}\label{s:e-u-m}

Under suitable assumptions, one can prove that the \emph{dual transfer operator} associated to an IFS is contracting in some Wasserstein distance. The completeness of the Wasserstein spaces thus makes it easy to prove existence and uniqueness of a stationary measure in $\proba_q(X)$ (some technicalities are needed to prove uniqueness on the whole of $\proba(X)$; Huntchinson restricts to compactly supported stationary measures). Later we shall also use this contraction property to get Lipschitz continuity of the stationary measure under perturbation of the IFS.

\subsection{Contracting dual operator}

We consider the ``dual transfer operator''
$\op{L}^*$ defined on $\proba(X)$ by
\[\op{L}^*\mu = \int (\phi_i)_*\mu \dd\eta(i),\]
i.e. $\op{L}^*\mu$ is the law of $\rv{X}_{n+1}$ if $(\rv{X}_n)_n$ is a Markov chain jumping from $x$ to $\phi_i(x)$ with probability $\dd\eta(i)$ and $\rv{X}_n\sim \mu$. A stationary measure is precisely a $\mu\in\proba(X)$ such that $\op{L}^*\mu=\mu$.

The ``transfer operator''  (also known as the Markov operator), acting for example on the space of bounded measurable functions $X\to \mathbb{R}$, is defined by
\[\op{L}f (x) = \int f\circ \phi_i(x) \dd \eta(i). \]
It is a positive operator fixing each constant function, so that when $a\le f\le b$ with $a,b\in\mathbb{R}$, then also $a \le \op{L}f \le b$. The duality relation between $\op{L}$ and $\op{L}^*$ is
\[\int f \dd(\op{L}^*\mu) = \int \op{L} f \dd\mu\]
and is a direct consequence of Fubini's theorem.

The dual transfer operator has a natural extension to couplings, which we denote in the same way: given $\gamma\in\Gamma(\mu_0,\mu_1)$, we define
\[ \op{L}^*\gamma = \int (\phi_i\times \phi_i)_*\gamma \dd\eta(i) \in\Gamma(\op{L}^*\mu_0,\op{L}^*\mu_1).\]

We first recover the following slight extension of the main results of \cite{barnsley1988newclass}, the difference being that only finite index sets were considered there. The first hypothesis requires the IFS to be contracting on $L^q$ average, while the second ensures that the maps do not shift some point too much (this was automatically granted in the finite index case, see Section \ref{s:tail} for examples showing the importance of this second hypothesis).
\begin{theo}[variant of Barnsley, Elton 1988]\label{th:main-stationary}
Let $(\Phi,\eta)$ be a IFS on a complete metric space $(X,d)$ and fix any $x_0\in X$. Assume that for some $q>0$, $A>0$, $\rho\in(0,1)$ the following holds:
\begin{align}
   \int d(\phi_i(x),\phi_i(y))^q \dd\eta(i) &\le \rho\, d(x,y)^q \qquad \forall x,y \in X  \label{e:thA-1}\\
  \int d(x_0,\phi_i(x_0))^q \dd\eta(i) &\le A.  \label{e:thA-2}
\end{align}
Then $(\Phi,\eta)$ has a unique stationary measure $\mu\in\proba(X)$; moreover $\mu$ has finite $q$th moment:
\[ m_{x_0}^q(\mu) \le \begin{dcases*}
   \frac{A}{1-\rho} & when $q\le 1$ \\[2\jot]
   \frac{A}{\big(1-\rho^{\frac1q}\big)^q} & when $q\ge 1$. \end{dcases*} \]
\end{theo}

The proof is split into several lemmas. It has some similarity with the original proof of Barnsley and Elton, our point here being to show that the Wasserstein distances are quite convenient, and enable us to use the dual transfer operator throughout the proof without introducing backward iterations; see also \cite{iosifescu2009criticalSurvey} where a slightly different metric is used.

\begin{lemm}\label{l:contraction}
The dual transfer operator preserves $\proba_q(X)$ and is a contraction of ratio no more than $\bar \rho :=\rho^{\min\left(1,\frac1q\right)}$.
\end{lemm}

\begin{proof}
Let $\mu_0,\mu_1\in\proba_q(X)$ and choose an optimal coupling $\gamma\in \Gamma(\mu_0,\mu_1)$ for $\wass_q$. Then
\begin{align*}
\int d(x,y)^q \dd (\op{L}^*\gamma)(x,y) &= \iint d(\phi_i(x),\phi_i(y))^q \dd\eta(i) \dd\gamma(x,y) \\
  &\le \rho\int  d(x,y)^q \dd\gamma(x,y)
\end{align*}
so that $\wass_q(\op{L}^*\mu_0,\op{L}^*\mu_1) \le \bar\rho \wass_q(\mu_0,\mu_1)$.

In particular, all elements of $\proba_q(X)$ are sent a finite $\wass_q$ distance from $\op{L}^*\delta_{x_0}$ and we only have left to prove that $\op{L}^*\delta_{x_0}\in\proba_q(X)$, which follows from \eqref{e:thA-2}:
\[
\cost_q(\delta_{x_0},\op{L}^*\delta_{x_0}) = \int d(x,y)^q \dd(\delta_{x_0}\otimes \op{L}^*\delta_{x_0})(x,y)
  =\int d(x_0,\phi_i(x_0)) \dd\eta(i) \le A.
\]
\end{proof}

\begin{lemm}\label{l:fixed}
There is a unique stationary measure in $\proba_q(X)$. Moreover for all $\nu\in\proba_q(X)$, we have $\op{L}^{*k}\nu \to \mu$ in the distance $\wass_q$, exponentially fast.
\end{lemm}

\begin{proof}
Follows from Lemma \ref{l:contraction} and the Banach fixed point theorem.
\end{proof}

\begin{lemm}\label{l:functions}
$\op{L}^k f(x) \to \int f \dd\mu$ for all continuous bounded functions $f:X\to\mathbb{R}$ and for all $x\in X$.
\end{lemm}

\begin{proof}
Since $\delta_x\in\proba_q(X)$,
\[ \op{L}^k f(x) = \int \op{L}^k f \dd\delta_x = \int f \dd(\op{L}^{*k}\delta_x) \to \int f \dd\mu.\]
\end{proof}

\begin{lemm}\label{l:finite-moment}
Every stationary measure has finite $q$-th moment, therefore there is a unique stationary measure in $\proba(X)$.
\end{lemm}

\begin{proof}
For each $n\in\mathbb{N}$, define a continuous bounded function by \[f_n(x) = \min(d(x_0,x)^q,n).\]
Let $\mu'\in\proba(X)$ be a stationary measure, the $q$-th moment of which we do not assume to be finite. For all $n,k\in \mathbb{N}$,
\[\int f_n \dd\mu' = \int f_n \dd(\op{L}^{*k}\mu') = \int \op{L}^k f_n \dd\mu'.\]
Since $\op{L}^k f_n$ is bounded between $0$ and $n$ for all $k$, we can apply the dominated convergence theorem as $k\to\infty$, so that by Lemma \ref{l:functions}
\[\int f_n \dd\mu' = \int \lim_{k\to\infty} \op{L}^k f_n \dd\mu' = \int \Big(\int f_n \dd\mu\Big) \dd\mu' = \int f_n \dd\mu \le \int d(x_0,x)^q \dd\mu < \infty.\]
The monotone convergence theorem applied to $f_n$ as $n\to\infty$ then shows that 
\[\int d(x_0,x)^q \dd\mu' \le  \int d(x_0,x)^q \dd\mu < \infty,\] so that $\mu'\in\proba_q(X)$. Since both $\mu$ and $\mu'$ are stationary measures of finite $q$-th moment, Lemma \ref{l:fixed} shows that $\mu'=\mu$.
\end{proof}

\begin{lemm}\label{l:moment}
The unique stationary measure $\mu$ satisfies $m_{x_0}^q(\mu) \le A/(1-\bar\rho)^{\max(1,q)}$.
\end{lemm}

\begin{proof}
Setting $\bar A = A^{\min\left(1,\frac1q\right)}$ and using $\op{L}^*\mu=\mu$ we get:
\begin{align*}
\wass_q(\delta_{x_0},\mu) &\le \wass_q(\delta_{x_0},\op{L}^*\delta_{x_0}) + \wass_q(\op{L}^*\delta_{x_0},\op{L}^*\mu) \\
  &\le \bar A + \bar\rho \wass_q(\delta_{x_0},\mu)\\
(1-\bar\rho)\wass_q(\delta_{x_0},\mu) &\le \bar A.
\end{align*}
If $q\le 1$, $\bar A=A$ and $\bar\rho=\rho$; we get $\int d(x_0,x)^q \dd\mu \le A/(1-\rho)$. If $q\ge 1$; $\bar A = A^{\frac1q}$ and $\bar\rho = \rho^{\frac1q}$ and we get
$\big(\int d(x_0,x)^q \dd\mu\big)^{\frac1q} \le A^{\frac1q}/(1-\rho^{\frac1q})$.
\end{proof}

Theorem \ref{th:main-stationary} follows at once from Lemmas \ref{l:fixed}, \ref{l:finite-moment} and \ref{l:moment}. 

As an illustration, let us consider a simple case studied for example in \cite{bergelson2006affine, anckar2016fine}. For each $\omega=(i_1,\dots,i_n)\in \{0,1\}^n$ we set $\phi_\omega = \phi_{i_1}\circ \phi_{i_2} \circ \cdots \phi_{i_n}$ (here the order of composition, forward or backward, has no particular relevance).
\begin{coro}
Let $a\in(0,1)$ and $b\in(1,\frac1a)$. The IFS on the line given by
\begin{equation}\begin{split}
I = \{0,1\},\qquad  &  \eta_p(\{0\}) = \frac12   \qquad  \eta_p(\{1\}) = \frac12 \\
 & \phi_0(x) = a x  \qquad  \phi_1(x) = b x + 1 \qquad \forall x \in \mathbb{R}
\end{split}\end{equation}
has a unique stationary measure $\mu$, which has finite moments of all orders $q\in(0,q_0)$ where $q_0$ is the unique solution in $(0,+\infty)$ of $a^{q_0}+b^{q_0}=2$. More precisely
\[m_0^q(\mu) \le \frac{2}{2-(a^q+b^q)}.\]
Moreover, for all $q\in(0,\min(1,q_0))$, all $q$-H\"older-continuous function $f:\mathbb{R}_+\to \mathbb{R}$ and all $x_0$,
\begin{equation}
\Big\lvert \frac{1}{2^n} \sum_{\omega \in \{0,1\}^n} f(\phi_\omega(x_0)) -\int f\dd\mu \Big\rvert \le C \rho^n
\label{e:affine-cv}
\end{equation}
where $\rho = \frac12(a^q+b^q)<1$ and $C=\Hol_q(f)/(1-\rho)$.
\end{coro}
We write $\Hol_q(f)$ for the least possible Hölder constant of $f$; note that we ask $f$ to be \emph{globally} $q$-H\"older, implying it has a growth at infinity of the order of $x^q$ at most. Of course, unbalanced versions of this example (i.e. with $\eta(\{0\})\neq\eta(\{1\})$) can be studied in the same way.

\begin{proof}
First note that the function $q\mapsto a^q+b^q$ is convex, and the assumptions ensure that it is decreasing on some interval $(0,q_1)$ and goes to $+\infty$ when $q\to+\infty$, so that this function takes the value $2$ at exactly two points, $0$ and $q_0$.

For all $q>0$, we have \eqref{e:thA-2} with $A=1$ and
\[\int \lvert \phi_i(x)-\phi_i(y)\rvert^q \dd\eta(i) = \frac12(a^q+b^q) \lvert x-y\rvert^q\]
so that when $q< q_0$, \eqref{e:thA-1} is satisfied with $\rho = \frac12(a^q+b^q)$. The claim on existence, uniqueness and moments of $\mu$ thus follows from Theorem \ref{th:main-stationary}.

The convergence of empirical averages of $f$ toward its integral with respect to $\mu$ follows from Lemma \ref{l:contraction}, observing
\[\frac{1}{2^n} \sum_{\omega \in \{0,1\}^n} f(\phi_\omega(x_0)) = \op{L}^n f(x_0).\]
Indeed as in the proof of Lemma \ref{l:functions} we have
\begin{align*}
\big\lvert \op{L}^n f(x_0) - \int f\dd\mu \big\rvert &= \big\lvert\int f \dd(\op{L}^{*n}\delta_{x_0}) - \int f\dd\mu \big\rvert \\
  &\le \Hol_q(f) \wass_q(\op{L}^{*n}\delta_{x_0},\mu) \\
  &\le \Hol_q(f) \rho^n \wass_q(\delta_{x_0},\mu).
\end{align*}
and $\wass_q(\delta_{x_0},\mu) = m_{x_0}^q(\mu) \le \frac{1}{1-\rho}$.
\end{proof}

\subsection{Some simple tools}

We present here a few statements that prove convenient for applications of Theorem \ref{th:main-stationary}.

\begin{lemm}\label{p:application}
If $(\Phi,\eta)$ satisfies \eqref{e:thA-1} and \eqref{e:thA-2}, then for all $q'\in(0,q)$ it also satisfies them with constants $q'$, $\rho':=\rho^{q'/q}$ and $A':=A^{q'/q}$.
\end{lemm}

\begin{proof}
Follows from the Jensen inequality applied to the concave function $r\mapsto r^{q'/q}$.
\end{proof}

The next two complementary results enable us to reduce the conditions of Theorem \ref{th:main-stationary} to other ``contracting on average'' hypotheses, including the one used for example in \cite{diaconis1999iterated}. Similar lemmas can be found in \cite{barnsley1988newclass} in the case of finite index set.

\begin{lemm}
If $(\Phi,\eta)$ satisfies 
\[\int \log \Lip(\phi_i) \dd \eta(i) <0 \quad\text{and}\quad \exists p>0, \int \Lip(\phi_i)^p \dd\eta(i)<+\infty,\]
then there exists $q>0,\rho\in (0,1)$ such that \eqref{e:thA-1} holds.
\end{lemm}

\begin{proof}
The idea is simply to differentiate $\int \Lip(\phi_i)^t \dd\eta(i)$ with respect to $t$ at $t=0$; we shall use truncation to differentiate under the integral sign.

For all $n\in\mathbb{N}$, consider the functions $f_n : I \to \mathbb{R}$ and $F_n:(0,p]\to\mathbb{R}$ defined by
\[f_n(i) = \max(\Lip(\phi_i),1/n) \quad\text{and}\quad F_n(t) = \int f_n(i)^t \dd\eta(i). \]
Since $-\log n \le \log f_n(i) \le \frac1p\Lip(\phi_i)^p$ for all $i,n$, the functions $\log f_n$ are $\eta$-integrable. The monotone convergence theorem implies that
\[\int \log f_n(i) \dd\eta(i) \to \int \log \Lip(\phi_i) \dd\eta(i) \in [-\infty,0)\]
so that for some $n\in\mathbb{N}$ we have $\int \log f_n(i) \dd\eta(i)\in (-\infty,0)$.
Now $F_n(0)=1$ and for all $t\in[0,\frac p2]$:
\[ -\log(n) \max\big(1,f_n(i)^{\frac p2}\big) \le \frac{\dd}{\dd t} f_n(i)^t = \log f_n(i) \cdot f_n(i)^t \le \frac2p f_n(i)^{t+\frac p2} \]
so that $\frac{\dd}{\dd t} f_n(i)^t$ is $\eta$-integrable, uniformly in $t$ ($n$ being fixed above). The function $F_n$ is thus differentiable on $[0,\frac p2]$, with $F_n'(0) = \int\log f_n(i) \dd\eta(i)<0$. We conclude that there is some $q\in (0,\frac p2)$ such that $F_n(q) \in (0,1)$. Now
\[ \int \Lip(\phi_i)^q \dd\eta(i) \le \int f_n(i)^q \dd\eta(i) = F_n(q) <1\]
which readily implies \eqref{e:thA-1}.
\end{proof}

\begin{lemm}
Assume that there exist $L\ge 1$ and $r\in(0,1)$ such that $\Lip(\phi_i)\le L$ for all $i\in I$ and 
\[\exp \Big( \int \log d(\phi_i(x),\phi_i(y)) \dd\eta(i) \Big) \le rd(x,y)\]
for all $x,y\in X$. Then there exists $q>0,\rho\in(0,1)$ such that \eqref{e:thA-1} holds.
\end{lemm}

\begin{proof}
Applying the second order Taylor formula with Lagrange remainder to $q\mapsto a^q$ ensures that 
\[a^q\le 1+q \log(a)  + C q^2\]
for all $a\in(0,L)$ and all $q\in(0,1]$, where $C=\sup\{\frac{a}{2}(\log a)^2 \colon x\in(0,L]\}$. Fix $x\neq y\in X$ and for all $i\in I$ apply this to $a=d(\phi_i(x),\phi_i(y))/d(x,y)$ and integrate with respect to $\eta$ to get:
\begin{align*}
\int \Big(\frac{d(\phi_i(x),\phi_i(y))}{d(x,y)}\Big)^q \dd\eta(i) 
  &\le 1+q\int \log\Big(\frac{d(\phi_i(x),\phi_i(y))}{d(x,y)}\Big) \dd\eta(i) + Cq^2 \\
  &\le 1+q\log(r)+Cq^2
\end{align*}
where $C$ is independent of $x,y$. Since $\log(r)<0$, we can find $q$ such that $1+q\log(r)+Cq^2<1$ and we are done.
\end{proof}

\subsection{A heavy tail of translations}\label{s:tail}

To illustrate the role of assumption \eqref{e:thA-2} in Theorem \ref{th:main-stationary} let us consider the following example on $[0,+\infty)$:
\begin{equation}\begin{split}
I = \mathbb{N},\qquad\qquad  &\qquad  \eta(\{n\}) = p_n, \\
\forall n>0, \forall x \in \mathbb{R}:\qquad & \phi_n(x) = x+n, \qquad \phi_0(x) = a x;
\end{split}
\label{e:tail}
\end{equation}
where $p_0>0$ and, of course, $p_n\ge 0$ and $\sum_{n\ge 0} p_n =1$.

We have good contraction properties: assumption \eqref{e:thA-1} is satisfied for any $q>0$ with $\rho = 1-(1-a^q)p_0$. As in the previous example, $\rho$ decreases from $1$ when $q\to0$ to $1-p_0$ when $q\to\infty$; however the translation part influences the moments of the stationary measure. By a direct application of Theorem \ref{th:main-stationary}, we get:
\begin{prop}
Let $q>0$; if $\sum n^q p_n<+\infty$ then \eqref{e:tail} has a unique stationary measure $\mu$, and $\mu$ has finite $q$-th moment; if $\sum n^q p_n = +\infty$, then any stationary measure of \eqref{e:tail} has infinite $q$-th moment. 
\end{prop}
Giving $(p_n)_n$ a heavy tail and taking $a\ll 1$ we get examples with very quick convergence in low-exponent Wasserstein distances but only few finite moments. This begs the question: what happens when $\sum n^q p_n=\infty$ for all $q>0$, e.g. when $p_n\sim c/n(\log n)^2$? Does there exist a stationary measure?

\section{Generalized moment estimates}\label{s:moments}

In some cases, the stationary measure of an IFS will not be compactly supported, but will have finite moments of all order; it then makes sense to develop tools to estimate exponential, sub-exponential or super-exponential moments.
In practice, the following simple result will be quite efficient.
\begin{prop}\label{p:gmoment}
Let $(\Phi,\eta)$ be an IFS on $X$ and $\varphi,\psi:X\to[0,+\infty)$ two functions that are bounded on every bounded subset of $X$ (usually, they will be of the form $h(d(x_0,\cdot))$ for some function $h$). Denote by $\op{L}$ the transfer operator, i.e. $\op{L}f(x) = \int f(\phi_i(x)) \dd\eta(i)$. Assume $\mu$ is a stationary measure for $(\Phi,\eta)$ and $\int \psi \dd\mu <+\infty$. 

If there exist $\theta\in(0,1)$ and $B\ge0$ such that for all $x\in X$, $\op{L}\varphi(x) \le \theta \varphi(x) + B\psi$, then $\int \varphi \dd\mu <+\infty$ and more precisely
\[\int\varphi \dd\mu \le \frac{B}{1-\theta} \int\psi\dd\mu.\]
\end{prop}

\begin{proof}
For all $m\in[0,+\infty)$, set $\varphi_m(x) = \min(\varphi(x),m)$.
Since $\mu$ is stationary, $\op{L}^*\mu=\mu$ and by positivity
\[\op{L}\varphi_m \le \min( \op{L}\varphi, m) \le \min(\theta\varphi + B \psi, m) \le \min(\theta\varphi,m) + B\psi = \theta \varphi_{\frac{m}{\theta}} + B\psi.\]
Applying duality we get:
\[\int \varphi_m \dd\mu = \int\varphi_m \dd\op{L}^*\mu 
  = \int \op{L}\varphi_m \dd\mu 
  \le \int \big( \theta\varphi_{\frac{m}{\theta}} + B\psi \big) \dd\mu\]
from which we deduce
\[ \int \big( \varphi_m-\theta\varphi_{\frac{m}{\theta}} \big) \dd\mu \le B\int \psi \dd\mu.\]
For each $x\in X$, the function $m\mapsto \varphi_m(x)-\theta\varphi_{\frac{m}{\theta}}(x)$ is non-decreasing and converges to $(1-\theta)\varphi(x)$. The monotone convergence theorem ensures that we can pass to the limit in the above inequality, leading precisely to the claimed inequality.
\end{proof}


Let us now treat a specific example for which to my knowledge no precise tail estimate has been derived yet.
For $p,a\in(0,1)$, consider the IFS given by
\begin{equation}\begin{split}
I = \{0,1\},\qquad  &  \eta(\{0\}) = p   \qquad  \eta(\{1\}) = 1-p \\
 & \phi_0(x) = a x  \qquad  \phi_1(x) = x+1 \qquad \forall x \in \mathbb{R}.
\end{split}\end{equation}
When $\phi_0$ and $\phi_1$ are seen as M\"obius transformations (i.e. extended as homography of the real projective line, with $\infty$ as a fixed point) or as hyperbolic isometries (i.e. extended to the Poincaré upper half plane of $\mathbb{C}$ with its hyperbolic metric), $\phi_0$ is \emph{hyperbolic} (two fixed points on the projective line, one attractive and one repulsive) while $\phi_1$ is \emph{parabolic} (a single fixed point on the projective line, which is repulsive on one side and attractive on the other side). This example is interesting in particular because it is not uniformly contracting, and cannot be made so in any set of coordinates because of the parabolic fixed point at infinity. Moreover it does not fit into Kesten and Goldie's framework for a polynomial tail, the parabolic map sits at the frontier between contraction (then the support would be bounded) and dilation (then only some moments would be finite, and a polynomial tail would be a possibility).

Let $\mu_{a,p}$ denote the unique stationary measure of this IFS; it is concentrated on $[0,+\infty)$. Its local structure has for example been studied in \cite{nicol2002fine, anckar2016fine}. Theorem \ref{th:main-stationary} shows that $\mu_{a,p}$ has finite moments of all orders. Indeed \eqref{e:thA-1} and \eqref{e:thA-2} are satisfied for all $q\ge 1$ with $x_0=0$, $\rho = 1-p + pa^q$ (which is less than $1$ since $a<1$), $A=1-p$. Since the maps of the IFS are affine, the transfer operator will behave very nicely with exponentials, and we get an explicit formula for the exponential moments.
\begin{coro}\label{c:HPmoment}
For all $p,a\in(0,1)$ and all $b < \log\frac{1}{1-p}$,
\[\int e^{bx} \dd\mu_{a,p}(x) = \prod_{k=0}^\infty \frac{p}{1-(1-p) e^{a^kb}} < \infty.\]
In particular there is a constant $C(a,p)$ independent of $b$ such that
\[\int e^{bx} \dd\mu_{a,p}(x)  \le \frac{C(a,p)}{1-(1-p)e^b}.\]
\end{coro}

\begin{proof}
Let $\op{L}$ be the dual transfer operator and $\varphi_b(x) = e^{bx}$; then
\begin{equation}
\op{L}\varphi_b(x) = p (\varphi_b(x))^a + (1-p) e^b \varphi_b(x).
\label{e:Lphi}
\end{equation}
Pick any $\theta$ strictly between $(1-p)e^b$ and $1$: for some $C>0$ (which could be computed explicitly), $\op{L}\varphi_b(x) \le \theta \varphi_b(x)$ whenever $x\ge C$. Since $\varphi_b$ is increasing, $\op{L}\varphi_b \le \theta\varphi_b + e^{bC}$ and we can thus apply Proposition \ref{p:gmoment} with $\psi \equiv 1$, obtaining
\[\int e^{bx} \dd\mu_{a,p} \le \frac{e^{bC}}{1-\theta}  <\infty.\]
This shows that $\varphi_b$ and $\varphi_{ab}$ are integrable with respect to $\mu_{a,p}$; using
\[\int\varphi_b \dd\mu_{a,p} = \int \varphi_b \dd \op{L}^*\mu_{a,p} = \int \op{L}\varphi_b \dd\mu_{a,p}\]
and \eqref{e:Lphi} we get
\begin{align*}
\int\varphi_b \dd\mu_{a,p} &= p\int \varphi_{ab} \dd\mu_{a,p} +(1-p)e^b \int\varphi_b\dd\mu_{a,p} \\
\int\varphi_b \dd\mu_{a,p} &= \frac{p}{1-(1-p)e^b}\int \varphi_{ab} \dd\mu_{a,p}
\end{align*}
Applying this equality to $\varphi_{a^k b}$, an induction yields
\[\int\varphi_b \dd\mu_{a,p} = \bigg(\prod_{k=0}^{K-1} \frac{p}{1-(1-p) e^{a^kb}}\bigg) \int \varphi_{a^K b} \dd\mu_{a,p}\]
for all $K\in\mathbb{N}$. We only have left to let $K\to\infty$: indeed $\varphi_{a^Kb}\to 1$ and the monotone convergence theorem gives the desired formula. 

The upper bound is then obtained by using $a^kb\le a^k\log\frac{1}{1-p}$ and setting \[C(a,p) = p\prod_{k=1}^\infty \frac{p}{1-(1-p)^{1-a^k}};\]
observe for the convergence that the logarithm of the $k$-th factor is asymptotic to $\frac{1-p}{p}(\log \frac{1}{1-p})a^k$.
\end{proof}

We obtain from this an exponential tail estimate for $\mu_{a,p}$, sharp up to a linear factor.
\begin{coro}\label{c:mainHP}
For all $p,a\in(0,1)$, there exist $c,C>0$ such that for all $t\ge 1$:
\[ c(1-p)^t  \le  \mu_{a,p}\big([t,+\infty) \big)  \le  Ct (1-p)^t.\]
\end{coro}

\begin{proof}
The lower bound follows from the simple observation that for all $n\in\mathbb{N}$ we have $\mu_{a,p}([n+1,n+2]) \ge (1-p) \mu_{a,p}([n, n+1])$, thus
$\mu_{a,p}([n,n+1]) \ge (1-p)^n \mu_{a,p}([0,1])$. Since $\mu_{a,p}$ is a probability measure, there is some $n\in\mathbb{N}$ such that $\mu_{a,p}([0,a^{-n}])>0$, and $\mu_{a,p}([0,1])\ge p^n \mu_{a,p}([0,a^{-n}])>0$.

The upper bound follows from Corollary \ref{c:HPmoment} and Chebyshev's inequality: for all $t>0$ and all $b<\log\frac{1}{1-p}$, 
\[ \mu_{a,p}([t,+\infty)) \le C(a,p)\frac{ e^{-bt}}{1-(1-p)e^b}.\]
Given $t$, we can choose $b$ in order to optimize the above inequality. An elementary computation leads to take 
\[e^b = \frac{1}{1-p} \cdot \frac{t}{1+t},\]
yielding the bound $\mu_{a,p}([t,+\infty)) \le C(a,p) p (1+\frac1t)^t(1+t)(1-p)^t = O\big(t(1-p)^t\big)$.
\end{proof}

\begin{prob}
Find an asymptotic for the tail of $\mu_{a,p}$, in the spirit of \cite{kesten1973random} and \cite{goldie1991implicit}.
 Note these works give (in a different context) a precise asymptotic $\mu([t,+\infty)) = f(t)+o(f(t))$ with $f$ a polynomial function; this can thus be written $\mu([t,+\infty)) = f(t+o(t))$. While Corollary \ref{c:mainHP} already gives an estimation of the form $\mu_{a,p}([t,+\infty)) = g(t+o(t))$, it might be difficult to get $\mu_{a,p}([t,+\infty))=g(t)+o(g(t))$ since $g$ is exponential and the sensitivity on $t$ is thus strong.
\end{prob}

\section{Convergence rate for the empirical measure}\label{s:empirical}

We now turn to the ``chaos game'' \cite{barnsley1988book}:
$(\rv{X}_k)_{k\in\mathbb{N}}$ is a Markov chain obtained by choosing random indices $(\rv{i}_k)_{k\ge 1}$ independently with law $\eta$, and setting $\rv{X}_{k}=\phi_{\rv{i}_k}(\rv{X}_{k-1})$; we shall say that $(\rv{X}_k)_{k\in\mathbb{N}}$ is \emph{driven} by $(\Phi,\eta)$. Quantitative results on the convergence of the laws of the $\rv{x}_k$ are abundant; formula \eqref{e:affine-cv} and Theorem \ref{th:skew} are examples, see also \cite{madras2010quantitative} for IFS satisfying a weak, local contraction assumption. We are interested here in the behavior of the \emph{empirical averages} (also known as \emph{Birkhoff sums} in the field of dynamical systems)
\[\frac1n\sum_{k=1}^n f(\rv{X}_k) =: \hat\mu_n(f)\]
for a given suitable `observable'' $f$, and of the empirical measure $\hat\mu_n$. The ergodic theorem of Elton \cite{elton1990ergodic} states that, under hypotheses similar to those of Theorem \ref{th:main-stationary}, $\hat\mu_n$ converges almost surely to the stationary measure $\mu$ in the weak-$*$ topology. We shall be interested in \emph{quantitative} ergodicity, i.e. in giving explicit estimates on the rate of convergence. For this, we will have to restrict to observables $f$ with some regularity, and slower rates are expected for the convergence of $\hat\mu_n$ in the Wasserstein metric than for the empirical average of a single function. Indeed, the former gives a simultaneous control for \emph{all} function in a certain class, and it is quite likely that the empirical averages of \emph{some} of these deviate from the standard long-term behavior. However, by asking for enough regularity we will be able to get a uniform rate close to $1/\sqrt{n}$, which in general cannot be surpassed even for a single $f$ because of the Central Limit Theorem.

Before turning to this, let us mention that while $\hat\mu_n(f)$ cannot in general converge to $\mu(f)$ at a faster rate than $1/\sqrt{n}$, one can get strong \emph{concentration} results, i.e. prove that $\hat\mu_n(f)$ is very likely to be very close to its expectation. Indeed, in a very general setting Ollivier introduced in \cite{ollivier2009ricci} Markov chains of \emph{positive Ricci curvature}. The inspiration comes from Riemannian geometry, where positive Ricci curvature can roughly be translated to the following property: given two points $x,y$, the uniform measures on the balls $B(x,r), B(y,r)$ are closer one to another (in  Wasserstein distance) than $x$ to $y$, by a factor $\rho<1$. On such a manifold, the Markov chain that jumps from $x$ to a uniform random point in $B(x,r)$ will thus have a unique stationary measure with exponential convergence. Ollivier's definition simply generalizes this to arbitrary Markov chains on metric spaces. That the dual transfer operator is a contraction in the Wasserstein distance $\wass_q$ for some $q\in (0,1]$ is a property equivalent to $(\rv{X}_k)_{k\in\mathbb{N}}$ having positive Ricci curvature in the sense of Ollivier on the space $(X,d^q)$, and implies strong concentration properties of the empirical averages $\hat\mu_n(f)$ whenever $f:X\to \mathbb{R}$ is a $q$-H\"older function; see \cite{joulin-ollivier2010curvature} for effective and completely explicit results (that depends on many specific quantities that may vary between examples).

%
%
%

Let us consider the following metrics between measures defined on $\mathbb{R}^d$:
\[\lVert \nu_0-\nu_1\lVert_{\Cku{s}} = \sup_{f\in\Cku{s}} \big\lvert \int f\dd\nu_0 - \int f\dd\nu_1 \big\rvert\]
where $s$ is any positive integer and $\Cku{s}$ is the set of $\mathcal{C}^{s-1}$ functions $\mathbb{R}^d\to\mathbb{R}$ with all derivatives not greater than $1$, and with their derivatives of order $(s-1)$ $1$-Lipschitz. In particular, $\lVert \cdot \rVert_{\Cku{1}} = \wass_1$. The following result shows that we can control the empirical averages of all observables of $\Cku{s}$ simultaneously, with very good bounds when $s$ is large enough compared to the dimension.
\begin{theo}\label{th:Markov}
Let $(\Phi,\eta)$ be an IFS on $\mathbb{R}^d$ satisfying \eqref{e:thA-1} and \eqref{e:thA-2} with $q=1$ and preserving a compact domain $D$ of $\mathbb{R}^d$.

 Let $(\rv{X}_k)_{k\in\mathbb{N}}$ be a Markov chain driven by $(\Phi,\eta)$, with $\rv{X}_0\in D$, and consider the empirical measures
\[\hat\mu_n := \frac1n \sum_{k=1}^n \delta_{\rv{X}_k}.\]

Then there exists a constant $C>0$ such that for all $n\in\mathbb{N}$,
\begin{equation}
\esp\big[ \lVert \hat\mu_n - \mu \rVert_{\Cku{s}} \big] \le C \begin{dcases*}
\frac{(\log n)^{\frac{d}{2s+1}}}{\sqrt{n}} & when $2s > d$\\[2\jot]
\frac{\log n}{\sqrt{n}} & when $2s=d$ \\[2\jot]
\frac{(\log n)^{d-2s+\frac sd}}{n^{\frac sd}} & when $2s < d$.
\end{dcases*}
\label{eq:theo-Markov}
\end{equation}
\end{theo}

\begin{proof}
Formula \eqref{eq:theo-Markov} is the conclusion of Theorem A in \cite{kloeckner2018empirical}, whose hypotheses are compactness of the domain and exponential contraction of the Markov chain in the metric $\wass_1$, which follows from Lemma \ref{l:contraction} with $q=1$.
\end{proof}

\begin{rema}
These rates cannot be improved, except possibly for the logarithm factors (see \cite{kloeckner2018empirical} for this and other considerations, including concentration bounds $\mathbb{P}[\lVert \hat\mu_n - \mu \rVert_{\Cku{s}} > M_n]\le \varepsilon_n$ for appropriate rates $M_n,\varepsilon_n$).
\end{rema}

\begin{rema}
If for some $\rho\in(0,1)$ every map $\phi_i$ of $\Phi$ is $\rho$-Lipschitz and $\{\lVert \phi_i(0) \rVert \colon i\in I\}$ is bounded, then there is a compact domain preserved by $\Phi$, so that the hypotheses of Theorem \ref{th:Markov} are satisfied. Indeed let $\Delta=\sup \{\lVert \phi_i(0) \rVert \colon i\in I\}$ and let $B$ be the ball $B$ of center $0$ and radius $R = \Delta/(1-\rho)$. Whenever $\lVert x\rVert \le R$, 
\[\lVert \phi_i(x) \rVert \le \lVert \phi_i(x) - \phi_i(0)\rVert + \lVert \phi_i(0) \rVert \le \rho \lVert x\rVert + \Delta \le R\]
so that $\phi_i(B)\subset B$ for all $i$.
\end{rema}

\begin{rema}
Relaxing the assumption to any number $q\in (0,1]$ is possible, but needs some adaptation from \cite{kloeckner2018empirical}. Theorem A from there asks contraction of $\op{L}^*$ in the $\wass_1$ metric, but the proof actually first reduces to contraction in the $\wass_\alpha$ for some $\alpha\in(0,1)$. Later $\alpha$ is optimized, but the optimal values goes to zero when $n\to\infty$, so only the value of the constant $C$ would affected.

Dispensing from the compactness assumption is also certainly possible, but the rates would necessarily be altered given the stationary measure might not have finite  moments of all orders (without some moment condition, one cannot expect even an optimal approximation by a discrete measure supported on $n$ point to achieve the rate $1/n^{\frac1d}$ in $\wass_1$ when $d>2$; see \cite{fournier2015rate} and \cite{dedecker2019moments} for rates of convergence of empirical measures under various moment assumptions).
\end{rema}

\section{Dependence of the stationary measure on the IFS and linear response}\label{s:response}

In this section we seek to quantify how close the stationary measures of two slightly different IFS  must be. To this end, we need to introduce a way to quantify the distance between IFS; it is both natural and effective to use an adaptation of Wasserstein distances. There are two points to consider in this adaptation: first, which metric to use for maps; second, how to take into account that we might consider IFS with different index sets. The second point is easily dealt with, by considering couplings $\gamma\in\proba(I_0\times I_1)$ of measures $\eta_0\in \proba(I_0)$ and $\eta_1\in\proba(I_1)$. There is much flexibility to address the first point; taking the uniform distance $d_\infty(\phi,\psi) = \sup_{x\in X} d(\phi(x),\psi(x))$ is ill suited to the non-compact case, as for example the map $x\mapsto ax+b$ acting on $\mathbb{R}$ would not depend continuously on the parameters $a,b$: changing $a$ the slightest bit would yield a map infinitely far from the original one. We therefore consider a pointed Lipschitz distance, notably suitable for Lipschitz IFS:
\begin{align*}
d_{x_0}(\phi,\psi) &:= \min\big\{\varepsilon\ge 0 \,\big|\, \forall x\in X \colon d(\phi(x),\psi(x)) \le \varepsilon + \varepsilon d(x,x_0)\} \\
  &= \sup_{x\in X} \frac{d(\phi(x),\psi(x))}{1+d(x,x_0)}.
\end{align*}
This defines a metric on the space of Lipschitz maps $X\to X$, and we construct from it the Wasserstein-like distance $\wass_{x_0,q}$ (possibly taking the value $\infty$) between IFS:
\begin{align*}
\cost_{x_0,q}\big((\Phi_0,\eta_0),(\Phi_1,\eta_1)\big) &:= \inf_{\gamma\in \Gamma(\eta_0,\eta_1)} \int d_{x_0}(\phi_i,\psi_j)^q \dd\gamma(i,j) \\
\wass_{x_0,q}\big((\Phi_0,\eta_0),(\Phi_1,\eta_1)\big) &:= \cost_{x_0,q}\big((\Phi_0,\eta_0),(\Phi_1,\eta_1)\big)^{\min(1,\frac1q)}.
\end{align*}

\subsection{Lipschitz regularity of the stationary measure}

\begin{theo}\label{th:main-close}
Consider two IFS $(\Phi_0,\eta_0)$ and $(\Phi_1,\eta_1)$ such that the first one satisfies \eqref{e:thA-1} and \eqref{e:thA-2} for some $q$, $\rho_0$, $A_0$ (and thus has a unique stationary measure $\mu_0$), and such that the second one has at least one stationary measure $\mu_1$ with finite $q$-th moment $m_{x_0}^q(\mu_1)$. Then 
\begin{align*}
\wass_q(\mu_0,\mu_1) &\le C\wass_{x_0,q}\big((\Phi_0,\eta_0),(\Phi_1,\eta_1)\big) \\
&\qquad \text{where }C = \begin{dcases*}
   \frac{1+ m_{x_0}^q(\mu_1)}{1-\rho_0} & when $q\le 1$ \\[2\jot]
   2^{1-\frac1q}\frac{(1+m_{x_0}^q(\mu_1))^{\frac1q} }{1-\rho_0^{\frac1q}} & when $q> 1$. \end{dcases*} 
\end{align*}
In particular, if we fix $\Phi=\Phi_0=\Phi_1$ and restrict to measures $\eta$ satisfying \eqref{e:thA-1} and \eqref{e:thA-2} the map $\eta\mapsto \mu$ (which is well-defined by Theorem \ref{th:main-stationary}) is locally Lipschitz.
\end{theo}

\begin{rema}
When the second IFS satisfies \eqref{e:thA-1} and \eqref{e:thA-2} with constants $q,\rho_1,A_1$, we can choose to apply the result after exchanging them to optimize; however using only the moment estimate of Theorem \ref{th:main-stationary} this is expected to only provide small improvements, since both spectral gaps (i.e. $1-\rho_0$ and $1-\rho_1$) are then involved in denominators.
\end{rema}

\begin{rema}
When the second ISF has several stationary measures with finite $q$-th moment, Theorem \ref{th:main-close} shows that they \emph{all} lie within small distance of $\mu_0$.
\end{rema}

\begin{rema}
For a fixed family of contractions $\Phi$ and varying probabilities $\eta_t$, Theorem \ref{th:main-close} gives a Lipschitz regularity that is stronger than what can be obtained in a similar case, for the thermodynamical formalism of expanding dynamical systems (see Section 6 of \cite{GKLM2018calculus}). When considering probabilities $\eta_t(i,x)$ that depend on the point $x$ and a time parameter $t$, the family $(\mu_t)$ of stationary measures should thus not be expected to be more than H\"older-continuous with respect to $t$.
\end{rema}

Let us now turn to the proof of Theorem \ref{th:main-close}. We denote by $\op{L}_k$ the transfer operator of the IFS $(\Phi_k,\eta_k)$ ($k\in \{0,1\}$).
\begin{lemm}\label{l:close-dtoperators}
For all IFS $(\Phi_0,\eta_0)$, $(\Phi_1,\eta_1)$ and all $\nu\in\proba_q(X)$
\begin{align*}
\wass_q(\op{L}_0^* \nu, \op{L}_1^*\nu) &\le D \wass_{x_0,q}\big((\Phi_0,\eta_0),(\Phi_1,\eta_1) \big)  \\
& \qquad\text{where } D = \begin{dcases*}
   1+m_{x_0}^q(\nu) & when $q\le 1$ \\[2\jot]
   2^{1-\frac1q}(1+m_{x_0}^q(\nu))^{\frac1q} & when $q> 1$. \end{dcases*}  
\end{align*}
\end{lemm}

\begin{proof}
If $\cost_{x_0,q}\big((\Phi_0,\eta_0),(\Phi_1,\eta_1)\big)=\infty$, the statement is emptily true. Assume otherwise, and let $\gamma\in \Gamma(\eta_0,\eta_1)$ be an optimal coupling. Let $\bar\nu = (\Id,\Id)_*\nu\in\proba(X\times X)$ be the trivial coupling of $\nu$ with itself. Then $\bar\gamma := \int (\phi_i,\psi_j)_*\bar\nu \dd\gamma(i,j)$ is a coupling of $\op{L}_0^*\nu$ and $\op{L}_1^*\nu$, so that
\begin{align*}
\cost_{q}(\op{L}_0^*\nu,\op{L}_1^*\nu)
  &\le \int d(x,y)^q \dd\bar\gamma(x,y) \\
  &= \iint d(\phi_i(x),\psi_j(y))^q \dd\bar\nu(x,y) \dd\gamma(i,j) \\
  &= \iint d(\phi_i(x),\psi_j(x))^q \dd\nu(x) \dd\gamma(i,j) \\
  &\le \iint \big(d_{x_0}(\phi_i,\psi_j)(1+d(x,x_0))\big)^q \dd\nu(x) \dd\gamma(i,j) \\
  &\le \int d_{x_0}(\phi_i,\psi_j)^q \dd\gamma(i,j) \int (1+d(x,x_0))^q\dd\nu(x).
\end{align*}
When $q\le 1$, using $(1+r)^q\le 1+r^q$ we obtain
\begin{align*}
\wass_{q}(\op{L}_0^*\nu,\op{L}_1^*\nu)
  &\le \wass_{x_0,q}\big((\Phi_0,\eta_0),(\Phi_1,\eta_1)\big) \int(1+d(x,x_0)^q)\dd\nu(x) \\
  &\le \wass_{x_0,q}\big((\Phi_0,\eta_0),(\Phi_1,\eta_1)\big)\big(1+m_{x_0}^q(\nu)\big)
\end{align*}
while when $q\ge 1$, using $(1+r)^q \le 2^{q-1}(1+r^q)$ we get
\begin{align*}
\wass_{q}(\op{L}_0^*\nu,\op{L}_1^*\nu)
  &\le\wass_{x_0,q}\big((\Phi_0,\eta_0),(\Phi_1,\eta_1)\big)\Big(2^{q-1}\int(1+d(x,x_0)^q)\dd\nu(x)\Big)^{\frac{1}{q}}\\
  &\le \wass_{x_0,q}\big((\Phi_0,\eta_0),(\Phi_1,\eta_1)\big) \cdot 2^{1-\frac1q}(1+m_{x_0}^q(\nu))^{\frac1q}.
\end{align*}
\end{proof}

\begin{proof}[Proof of Theorem \ref{th:main-close}]
Apply Lemma \ref{l:close-dtoperators} to $\nu=\mu_1$ and use that $\op{L}_0^*$ is a contraction (Lemma \ref{l:contraction}, recall $\bar\rho_0 = \rho_0^{\min(1,\frac1q)}$):
\begin{align*}
\wass_q(\mu_0,\mu_1) &\le \wass_q(\mu_0,\op{L}_0^*\mu_1) + \wass_q(\op{L}_0^*\mu_1,\mu_1) \\
  &= \wass_q(\op{L}_0^*\mu_0,\op{L}_0^*\mu_1) + \wass_q(\op{L}_0^*\mu_1,\op{L}_1^*\mu_1) \\
  &\le \bar\rho_0 \wass_q(\mu_0,\mu_1) + D\wass_{x_0,q}\big( (\Phi_0,\eta_0), (\Phi_1,\eta_1) \big) \\
\wass_q(\mu_0,\mu_1) &\le \frac{D}{1-\bar \rho_0} \wass_{x_0,q}\big( (\Phi_0,\eta_0), (\Phi_1,\eta_1) \big).
\end{align*}
\end{proof}

\subsection{Linear response}

The Rademacher theorem ensures that Lipschitz functions $[a,b]\to\mathbb{R}$ are differentiable Lebesgue almost-everywhere; a similar result has been proven by Ambrosio, Gigli and Savar\'e \cite{ambrosio2008gradient} for maps $[a,b]\to \proba_q(\mathbb{R}^n)$ for $q>1$, where one has of course to make precise what ``differentiable'' means. Together with Theorem \ref{th:main-close}, this provides a ``linear response formula'' in many cases. We use $u\cdot w$ to denote the scalar product of two vectors $u,w\in\mathbb{R}^n$.
\begin{coro}[Linear Response]\label{c:LinearResponse}
Let $(\Phi_t,\eta_t)_{t\in [a,b]}$ be a curve of IFS on $\mathbb{R}^n$ (endowed with the Euclidean metric, the origin $O$ serving as reference point), assume that for some $q > 1$,
\begin{enumerate}
\item\label{enumi:lip} $t\mapsto (\Phi_t,\eta_t)_{t}$ is Lipschitz in $\wass_{O,q}$,
\item there exist $\rho_+\in (0,1)$ and $A_+>0$ such that for all $t\in [a,b]$ the IFS $(\Phi_t,\eta_t)$ satisfies hypotheses \eqref{e:thA-1} and \eqref{e:thA-2} with parameters $q$, $\rho_t\le \rho_+$ and $A_t\le A_+$
\end{enumerate}
and let $\mu_t$ denote the unique stationary measure of $(\Phi_t,\eta_t)$. Then there exist a family $(v_t)_{t\in [a,b]}$ of measurable vector fields on $\mathbb{R}^n$ such that:
\begin{enumerate}[resume]
\item\label{enumi:gradient}  for Lebesgue almost all $t$, $\lVert v_{t}\rVert \in L^q(\mu_{t})$, and $\lvert v_{t}\rvert^{q-2} v_{t}$ can be approximated by gradients of smooth functions $\mathbb{R}^n\to \mathbb{R}$ in the $L^{q'}(\mu_t)$ norm where $q'=q/(q-1)$,
\item\label{enumi:diff} $\frac{\dd}{\dd t}\mu_t + \div(v_t \mu_t) = 0$ weakly on $\mathbb{R}\times \mathbb{R}^n$, i.e. for all smooth compactly supported $F:\mathbb{R}\times\mathbb{R}^n$:
\[\int_\mathbb{R}\int_{\mathbb{R}^n} \Big(\frac{\dd}{\dd t}F(t,x) + \grad_x F(t,x) \cdot v_t(x) \Big) \dd\mu_t\dd t = 0,\]
\item\label{enumi:diffW} for Lebesgue almost all $t_0$,
\[\wass_q\big(\mu_{t+\varepsilon}, (\Id+\varepsilon v_t)_*\mu_t \big) = o(\varepsilon). \]
\end{enumerate}
As a consequence of \ref{enumi:diffW}, at almost every $t_0\in [a,b]$, for all compactly supported smooth functions $f:\mathbb{R}^n\to\mathbb{R}$:
\[ \frac{\dd}{\dd t}\Big\rvert_{t=t_0} \int f \dd\mu_t  = \int \grad f \cdot v_{t_0} \dd\mu_{t_0}.\]
\end{coro}
Note that here $q>1$ is needed to ensure strict convexity in the optimal transport problem. Property \ref{enumi:gradient} may seem rather exotic; it is more easily explained when $q=2$: the approximation of $v_{t}$ by gradient of smooth functions is a way to formulate that the ``curl with respect to $\mu_{t}$'' of this vector field vanishes, which relates to optimality in an ``infinitesimal'' transport problem. On $\mathbb{R}$ it is vacuous but in higher dimension it is important as it ensures uniqueness of $v_t$.

The proof is only an application of some results in \cite{ambrosio2008gradient}, but we detail some classical arguments to better show how optimal transportation is related to linear response formulas.
\begin{proof}
By Theorem \ref{th:main-close}, the family of stationary measures $(\mu_t)_{t\in [a,b]}$ is Lipschitz, in particular absolutely continuous. Thus Theorem 8.3.1 of \cite{ambrosio2008gradient} applies, giving \ref{enumi:gradient} and \ref{enumi:diff} (note the formulation (8.1.4) for the interpretation of the continuity equation, and see above Definition 5.1.11 that $\mathrm{Cyl}(\mathbb{R}^n)$ is the space of smooth compactly supported functions).
Proposition 8.4.6 of \cite{ambrosio2008gradient} gives \ref{enumi:diffW} and we are left with proving the differentiation formula for $\int f\dd\mu_t$.

If we fix $f$, the weak derivative given in \ref{enumi:diff} is sufficient to obtain the derivative of $t\mapsto \mu_t(f)$ almost everywhere; but it could be that the negligible set for which the formula fails turns out to depend on $f$. We therefore use \ref{enumi:diffW}: fix any $t$ at which it holds and $f$ a smooth compactly supported function. By Jensen's inequality, since $q>1$ we have
$\wass_1(\mu_{t+\varepsilon},(\Id+\varepsilon v_t)_*\mu_t)=o(\varepsilon)$, and since $f$ is Lipschitz the dual formulation of the Wasserstein distance yields
\[\int f\dd\mu_{t+\varepsilon} = \int f(x+\varepsilon v_t(x)) \dd\mu_t +o(\varepsilon) \]
The second order Taylor formula ensures that $f(x+\varepsilon v_t(x)) = f(x)+\varepsilon v_t(x) \cdot \grad f(x) +O(\varepsilon^2 v_t(x)^2)$ where the implied constant is uniform in $x$. When $q\ge 2$, we thus get
\[\int f\dd\mu_{t+\varepsilon} = \int f\dd\mu_t +\varepsilon \int \grad f \cdot v_t \dd\mu_t + O\big(\varepsilon^2\int v_t^2 \dd\mu_t\big) + o(\varepsilon).\]
with $v_t\in L^2(\mu_t)$, giving the desired derivative at $t$. When $q\in(1,2)$, we argue as follows. 

Let $\alpha\in (1,3-\frac2q)$ and set $B_\varepsilon = \{x\in\mathbb{R}^n \mid \varepsilon^2v_t(x)^2 >\varepsilon^\alpha v_t(x)^q \}$ and $G_\varepsilon=\mathbb{R}^n\setminus B_\varepsilon$. By Chebyshev's inequality, 
\[ \mu_t(B_\varepsilon) = \mu_t(\{v^q > \varepsilon^{-(2-\alpha)\frac{q}{2-q}}\}) \le \varepsilon^{(2-\alpha)\frac{q}{2-q}}\int v_t^q \dd\mu_t = O(\varepsilon^\beta)=o(\varepsilon)\]
where $\beta=(2-\alpha)\frac{q}{2-q}>1$. It follows:
\begin{align*}
\int f\dd\mu_{t+\varepsilon} &= \int_{B_\varepsilon} f(x+\varepsilon v_t(x)) \dd\mu_t + \int_{G_\varepsilon} f(x+\varepsilon v_t(x)) \dd\mu_t +o(\varepsilon)\\
  &=  O(\mu_t(B_\varepsilon)) + \int_{G_\varepsilon} f \dd\mu_t + \varepsilon \int_{G_\varepsilon} \grad f\cdot v_t \dd\mu_t +  \varepsilon^\alpha \int_{G_\varepsilon} v_y^q\dd\mu_t +o(\varepsilon) \\
  &= \int f \dd\mu_t + \varepsilon \int \grad f\cdot v_t \dd\mu_t + O(\mu_t(B_\varepsilon)) + o(\varepsilon) \\
  &= \int f \dd\mu_t + \varepsilon \int \grad f\cdot v_t \dd\mu_t +o(\varepsilon)
\end{align*}
as desired.
\end{proof}

\subsection{The case of Bernoulli convolutions}

Corollary \ref{c:linearBernoulli} will follow from Corollary \ref{c:LinearResponse}. We take as reference point $x_0=O=0\in\mathbb{R}$; recall that the family $(\Phi^\lambda,\eta)$ of IFS defining Bernoulli convolutions is given in Section \ref{subs:Bernoulli}.

We start by the following explicit Lipschitz estimate.
\begin{prop}
For all $q\ge 1$ and all $\lambda,\lambda'\in(0,1)$:
\[\wass_q(\mu_\lambda,\mu_{\lambda'}) \le \frac{2}{1-\lambda} \lvert \lambda-\lambda'\rvert.\]
In particular, $\lambda\mapsto\mu_\lambda$ is Lipschitz in the metrics $\wass_q$ on each interval of the form $[0,1-\varepsilon]$ where $\varepsilon>0$.
\end{prop}
\begin{proof}
Given any $q\ge 1$ and $\lambda,\lambda'\in (0,1)$, 
\[d_{x_0}(\phi_0^\lambda,\phi_0^{\lambda'}) = d_{x_0}(\phi_1^\lambda,\phi_1^{\lambda'}) = \lvert \lambda-\lambda'\rvert.\] 
Considering the identity coupling of $\eta$ with itself, defined by $\gamma(\{(0,0)\}) = \gamma(\{(1,1)\})=\frac12$ and $\gamma(\{(0,1)\}) = \gamma(\{(1,0)\}) = 0$, we get
$\wass_{0,q}\big((\Phi^\lambda,\eta),(\Phi^{\lambda'},\eta)\Big) \le \lvert \lambda-\lambda'\rvert$; moreover $(\Phi^\lambda,\eta)$ satisfies \eqref{e:thA-1} and \eqref{e:thA-2} with constants $\rho=\lambda^q$ and $A=1$; and $m_{x_0}^q(\mu_{\lambda'})\le 1$. Theorem \ref{th:main-close} then ensures that
\[\wass_q(\mu_\lambda,\mu_{\lambda'}) \le \frac{2^{1-\frac1q}2^{\frac1q}}{1-\lambda}\lvert \lambda-\lambda'\rvert.\]
\end{proof}

\begin{proof}[Proof of Corollary \ref{c:linearBernoulli}]
Fix any $q>1$; as seen above, $\wass_{0,q}\big((\Phi^\lambda,\eta),(\Phi^{\lambda'},\eta)\Big) \le \lvert \lambda-\lambda'\rvert$ and in restriction to any interval $[\frac12,1-\varepsilon]$ with $\varepsilon>0$ we have \eqref{e:thA-1} and \eqref{e:thA-2} with uniform bounds $\rho_+=(1-\varepsilon)^q$ and $A_+=1$.

Corollary \ref{c:LinearResponse} provides us with a family $(v_\lambda)_{\lambda\in(0,1)}$ of vector fields on $\mathbb{R}$, which can be identified with $\lVert v_\lambda\rVert \in L^q(\mu_\lambda)$ functions, such that for Lebesgue-almost all $\lambda\in(0,1)$, $\wass_q(\mu_{\lambda+\varepsilon},(\Id+ \varepsilon v_\lambda)_*\mu_\lambda)=o(\varepsilon)$. When $\lambda>\frac12$, up to further restricting to a subset of full Lebesgue measure for the parameter $\lambda$, Solomyak's theorem \cite{solomyak1995erdos} ensures that $\mu_\lambda$ is absolutely continuous with respect to the Lebesgue measure; let us denote its density by $g_\lambda$. For almost all $\lambda\in(\frac12,1)$ and all smooth compactly supported test function $f:\mathbb{R}\to\mathbb{R}$ we get
\[\frac{\dd}{\dd t}\Big\rvert_{t=\lambda} \int f \dd\mu_t  = \int_0^1 f'(x) v_\lambda(x) g_\lambda(x) \dd x,\]
which is the desired formula with $w_\lambda=v_\lambda g_\lambda$. Moreover, for almost-all $\lambda>1/\sqrt{2}$, $g_\lambda$ is bounded (Corollary 1 in \cite{solomyak1995erdos}), implying $w_\lambda\in L^q(\mu_\lambda)$ and then $w_\lambda\in L^q([0,1])$.

Now, $w_\lambda$ seems to depend on the choice of $q$. But if $\tilde w_\lambda$ is another suitable choice for the same $\lambda$ (and possibly different $q$), extending both of them by $0$ outside $[0,1]$, we have $\int_{\mathbb{R}} f' w_\lambda = \int_{\mathbb{R}} f'\tilde w_\lambda$ for all test functions $f$. As a consequence the extensions of $w_\lambda$ and $\tilde w_\lambda$ must differ by a constant, and we thus must have $w_\lambda=\tilde w_\lambda$. It follows that there is a single $w_\lambda$, belonging to all $L^q([0,1])$ simultaneously.
\end{proof}

\section{Stationary measures beyond products}\label{s:beyond-products}

While the case of IFS as defined above, where the randomness is a consequence of a sequence of independent random variables of law $\eta\in\proba(I)$, is the most commonly studied, there has been great interest to generalize this setting. A first generalization is to replace the i.i.d. sequence by a stationary Markov chain; a further generalization is to draw the infinite word $\omega=\omega_0\dots\omega_k\dots$ randomly with law an arbitrary shift-invariant measure $\nu\in\proba(I^{\mathbb{N}})$ -- the case of IFS corresponding to the independent Bernoulli product $\nu=\eta^{\otimes\mathbb{N}}$; then one can consider a yet further generalization where the shift is replaced by an arbitrary measure-preserving dynamical system.

\subsection{Skew-products}

We still consider $(X,d)$ a complete metric space, and we additionally fix a standard measure space $(Y,\mathcal{A})$ (i.e. it is isomorphic to $[0,1]$ with its Borel $\sigma$-algebra) equipped with a probability measure $\nu$, and a $\nu$-preserving map $S:Y\to Y$. A \emph{skew-product map} over $S$ with fiber $X$ is a map
\begin{align*}
\Psi : X\times Y &\to X\times Y \\
      (x,y) &\mapsto (\psi_y(x), S(y))
\end{align*}
where $(x,y)\mapsto \psi_y(x)$ is a measurable map. While, as we have seen above, an IFS can be studied dynamically by looking at a random orbit $\rv{X}_0$, $\rv{X}_{n+1}=\phi_{\rv{I}_n}(\rv{X}_n)$ where $(\rv{I}_n)_{n\ge 1}$ are i.i.d. random variables of law $\eta$, in the present setting the corresponding random sequence of points is given by $\rv{X}_{n+1}=\psi_{S^n(\rv{Y})}(\rv{X}_n)$ where $\rv{Y}$ is a random element of $Y$ with law $\nu$, taking the place of the whole sequence $(\rv{I}_1,\rv{I}_2,\dots)$. In other words, IFS correspond to the particular case when $Y=I^{\mathbb{N}}$, $\nu=\eta^{\otimes\mathbb{N}}$, $S$ is the shift $y=(y_0,y_1,\dots) \mapsto S(y) = (y_1,y_2,\dots)$ and $\psi_y(x) = \phi_{y_0}(x)$.
Note that $\Psi$ carries the information of what are $X$, $Y$ and $S$; when we refer to this setting, we shall therefore call $(\Psi,\nu)$ a \emph{skew product}.

We denote by $\pi^X$, $\pi^Y$ the projection maps from $X\times Y$ to each factor; a measure $\mu\in\proba(X)$ is said to be a \emph{stationary measure} of the skew product $(\Psi,\nu)$ when there exists a measure $\hat\nu\in\proba(X\times Y)$ such that:
\[\hat\nu \text{ is }\Psi\text{-invariant}, \qquad \pi^{Y}_*\hat\nu = \nu, \quad\text{and}\quad \pi^{X}_*\hat\nu = \mu.\]
In the case of an IFS, this coincides with the previous definition of stationary measure. The measure $\hat\nu$ as above will be called a \emph{lift} of $\nu$. The basic question we want to address under specific assumptions is whether there exist a unique stationary measure; a positive answer will follow from the uniqueness of the lift of $\nu$.

%
%

\begin{defi}
We say that $\Psi$ \emph{contracts the fibers} whenever there exist $\rho\in(0,1)$ such that for all $y\in Y$, the map $\psi_y$ is $\rho$-Lipschitz. 

We say that $\Psi$ has \emph{bounded displacement} if for some $x_0\in X$, there exist an $A>0$ such that the set $d(x_0,\psi_y(x_0))\le A$ for all $y\in Y$.
\end{defi}
Observe that when $(\psi_y)_{y\in Y}$ is an equicontinuous family, e.g. when $\Psi$ contracts the fibers, in the definition of bounded displacement ``for some $x_0$'' could be equivalently replaced by ``for all $x_0$'' (up to changing the value of $A$).

The main result of this section is the following.
\begin{theo}\label{th:skew}
Let $\Psi$ be a skew-product map on $X\times Y$ that contracts the fibers and has bounded displacement. Each $S$-invariant $\nu\in\proba(Y)$ has a unique lift, and in particular the skew product $(\Psi,\nu)$ has a unique stationary measure $\mu$, which moreover has bounded support.

Let $(\rv{X}_k)_{k\in\mathbb{N}}$ be a stochastic process associated to $(\Psi,\nu)$  as above, with $\rv{X}_0$ independent from $\rv{Y}$ and of arbitrary law $\tilde\mu_0\in\proba_q(X)$ for some $q>0$, and let $\tilde\mu_k\in\proba(X)$ be the law of $\rv{X}_k$. Then for all $k\in\mathbb{N}$,
\[\wass_q(\tilde\mu_k,\mu)\le D \tilde\rho^k\]
where
\[\tilde\rho =  \rho^{\min(q,1)} \in (0,1), \qquad D =  m_{x_0}^q(\tilde\mu_0)^{\min(1,\frac1q)} + \big(\frac{A}{1-\rho}\big)^{\min(1,q)},\]
and $A,\rho$ are the constants in the bounded displacement and fiber contraction hypotheses.
\end{theo}

\subsection{Fiber-wise Wasserstein distance}

The main tool to prove Theorem \ref{th:skew} is a variation of Wasserstein distance that is adapted to a projection map and the inverse images of a given measure on its target space. This notion was at the heart of \cite{kloeckner2018shrinking}, from which we adapt the relevant definitions and properties. Theorem A from \cite{kloeckner2018shrinking} is not immediately applicable here since $X$ need not be compact, $\diam(\Psi^n(X\times \{y\}))$ might be infinite for all $n$, and $Y$ is not even a topological space; but the adaptation is relatively straightforward.

Fix any $\nu\in \proba(Y)$ and let $\proba^\nu := (\pi^Y_*)^{-1}(\nu) \subset \proba(X\times Y)$ be the \emph{fiber} of $\nu$, i.e, the set of measures on $X\times Y$ with second marginal equal to $\nu$. Recalling that we fixed a point $x_0\in X$, given any $\sigma\in \proba^\nu$ and $q>0$ we define its $q$-th moment by
\[m_{x_0}^q(\sigma) = \int d(x,x_0)^q \dd\sigma(x,y)\]
where the integral is over the whole product $X\times Y$ but distances are recorded only ``along the fibers'', i.e. over the $X$ factor. We let $\proba^\nu_q$ be the subset of $\proba^\nu$ consisting of measures of finite $q$-th moment (this set does not depend on $x_0$).

The product $X^2\times Y$ identifies with what was noted $\Delta_\pi$ in \cite{kloeckner2018shrinking} (pairs of point in the total space that project to the same point on the base $Y$); we consider the  maps
\[\pi_{02}:(x,x',y) \mapsto (x,y) \qquad \pi_{12}:(x,x',y)\mapsto (x',y) \qquad \pi_{2} : (x,x',y)\mapsto y.\]
For all $\sigma_0,\sigma_1\in\proba^\nu$ let $\Gamma^\nu(\sigma_0,\sigma_1) := \{\gamma\in\proba(X^2\times Y) \mid (\pi_{02*})\gamma=\sigma_0 \text{ and } (\pi_{12*})\gamma=\sigma_1\}$ (playing the role of $\Gamma_\pi$ in \cite{kloeckner2018shrinking}, where we had chosen to emphasize the projection map rather than the image measure) and define
\begin{align*}
\cost^\nu_q (\sigma_0,\sigma_1) &= \inf_{\gamma\in\Gamma^\nu(\sigma_0,\sigma_1)} \int d(x,x')^q \dd\gamma(x,x',y) \\
\wass^{\nu}_q (\sigma_0,\sigma_1) &= \big(\cost^\nu_q (\sigma_0,\sigma_1)\big)^{\min(1,\frac1q)}.\end{align*}

The following basic result is proven in the same way as in \cite{kloeckner2018shrinking}.
\begin{prop}\label{prop:vertical}
For all $\sigma_0,\sigma_1\in \proba^\nu$, the set $\Gamma^\nu(\sigma_0,\sigma_1)$ is non-empty. If the moments $m_{x_0}^q(\sigma_i)$ are finite for $i\in\{0,1\}$, then $\wass^{\nu}(\sigma_0,\sigma_1)<\infty$. Moreover, if $(\xi_y)_{y\in Y}$ and $(\zeta_y)_{y\in Y}$ are the disintegrations of $\sigma_0$ and $\sigma_1$ with respect to $\pi^Y$, then 
\begin{equation}
\wass_q^{\nu}(\sigma_0,\sigma_1) = 
\begin{dcases*}
\Big(\int \wass_q(\xi_y,\zeta_y)^q \dd\nu(y)\Big)^{\frac1q} 
    & when $q\ge 1$ \\[2\jot]
\int \wass_q(\xi_y,\zeta_y) \dd\nu(y)
    & when $q\le 1$.
\end{dcases*}
\label{e:vertical}
\end{equation}
Finally, $\wass^{\nu}_q$ is a complete metric on the set $\proba^\nu_q$.
\end{prop}

\begin{proof}
Let $(\xi_y)_{y\in Y}$ and $(\zeta_y)_{y\in Y}$ be the disintegrations of $\sigma_0$ and $\sigma_1$ with respect to $\pi^Y$ (identifying $X$ with the fibers of $\pi^Y$, $(\xi_y)_{y\in Y}$ is  thus a family of measures on $X$ characterized by $\int f(x,y) \dd\xi_y(x) \dd\nu(y) = \int f(x,y)\dd\sigma_0(x,y)$ for all continuous bounded functions $f:X\times Y\to\mathbb{R}$.)

From any measurable choice of $y\mapsto \gamma_y\in\Gamma(\xi_y,\zeta_y)$ (e.g. $\gamma_y = \xi_y\otimes\zeta_y$) we can build an element $\gamma$ of $\Gamma^\nu(\sigma_0,\sigma_1)$ by setting $\int f(x,x',y) \dd\gamma(x,x',y) = \iint f(x,x',y) \dd\gamma_y(x,x') \dd\nu(y)$. In particular $\Gamma^\nu(\sigma_0,\sigma_1)$ is non-empty.

Conversely, given any $\gamma\in \Gamma^\nu(\sigma_0,\sigma_1)$ its disintegration with respect to $\pi_2$ is a family $(\gamma_y)_{y\in Y}$ of measures on $X\times X$, and by testing $\gamma$ against integrands of the form $f(x)g(y)$ and $f(x')g(y)$ one sees that $\gamma_y\in\Gamma(\xi_y,\zeta_y)$ for $\nu$-almost all $y$.

Since $\int d(x,x')^q \dd\gamma(x,x',y) = \iint d(x,x')^q \dd\gamma_y(x,x') \dd\nu(y) \ge \int \cost_q(\xi_y,\zeta_y) \dd\check\mu(y)$,
taking an infimum we get $\cost^{\nu}_q(\mu_0,\mu_1)\ge \int \cost_q(\xi_y,\zeta_y)\dd\check\mu$.

For each $y$, the set of optimal transport plans from $\xi_y$ to $\zeta_y$ is compact (see e.g. the proof of Theorem 4.1 in \cite{Villani2009OldNew}), thus by the measurable selection theorem there is a measurable family $(\gamma_y)_{y\in Y}$ such that for $\nu$-almost all $y\in Y$, $\int d(x,x')^q \dd\gamma_y(x,x') = \cost_q(\xi_y,\zeta_y)$. It follows $\cost^{\nu}_q(\sigma_0,\sigma_1)\le \int \cost_q(\xi_y,\zeta_y)\dd\nu$ and \eqref{e:vertical} is proven.

To complete the proof, it remains to be seen that $\cost^{\nu}_q(\sigma_0,\sigma_1)<\infty$ and that $\wass^{\nu}_q$ is a metric making $\proba^\nu_q$ a complete space.
The triangular inequality follows from \eqref{e:vertical}, and then finiteness is obtained by observing 
\[\wass^\nu_q(\sigma_0,\sigma_1) \le \wass^\nu_q(\sigma_0,\delta_{x_0}\otimes \nu) + \wass^\nu_q(\delta_{x_0}\otimes \nu, \sigma_1) = m_{x_0}^q(\sigma_0) + m_{x_0}^q(\sigma_1).\]

Finally, The Riesz-Fischer Theorem for metric-space valued functions ensures that $\wass^\nu_q$ is a complete metric on $\proba^\nu_q$, seen \emph{via} disintegration as a closed subset of the space of maps $Y\to \proba_q(X)$.
\end{proof}

\subsection{Proof of Theorem \ref{th:skew}}

Let $\nu$ be any $S$-invariant probability measure on $Y$. We first observe that the fiber contraction and bounded displacement properties ensure that $\Psi_*$ preserves  $\proba^\nu_q$ for all $q$. 
These uniform assumptions also ensure that for some bounded set $B\subset X$, the set $B\times Y$ is an \emph{absorbing invariant set}, i.e. $\Psi(B\times Y)\subset B\times Y$ and for all $(x,y)\in X\times Y$ there is some $k\in\mathbb{N}$ such that  $\Psi^k(x,y)\in B\times Y$.  Let indeed $A>0$ be such that for all $y$, $d(x_0,\psi_y(x_0))\le A$, fix any $\varepsilon>0$, set $R=(1+\varepsilon)A/(1-\rho)$ and let $B=B(x_0,R)$ be the ball of center $x_0$ and radius $R$ in $X$; then for all $x,y\in X\times Y$ 
\begin{align*}
d(x_0,\psi_y(x)) &\le d(x_0,\psi_y(x_0)) + d(\psi_y(x_0),\psi_y(x)) \\
  &\le A+\rho d(x_0,x).
\end{align*}
When $x\in B$, the right-hand side is at most $A+\rho R = \frac{1+\varepsilon\rho}{1-\rho} A < R$, proving the $B\times Y$ is $\Psi$-invariant. When $x\notin B$, we have $A< \frac{1-\rho}{1+\varepsilon}d(x_0,x)$ and the right-hand side is at most 
\[\big(\frac{1-\rho}{1+\varepsilon}+\rho\big)d(x_0,x) = \frac{1+\varepsilon\rho}{1+\varepsilon}d(x_0,x)\]
where $\frac{1+\varepsilon\rho}{1+\varepsilon}<1$, proving the absorbing property with $k\simeq \log d(x_0,x)$.

Let $\sigma_0,\sigma_1\in\proba^\nu_q$. We consider the map $X^2\times Y \to X^2\times Y$ defined by
\[\Psi_2(x,x',y) = (\psi_y(x),\psi_y(x'),S(y)).\]
For $i\in\{0,1\}$ we have $\pi_{i2}\circ \Psi_2 = \Psi\circ \pi_{i2}$; as a consequence, for any $\gamma\in\Gamma^\nu(\sigma_0,\sigma_1)$ we have $\Psi_{2*}\gamma\in\Gamma^\nu(\Psi_*\sigma_0,\Psi_*\sigma_1)$. Observing
\begin{align*}
\int d(x,x')^q \dd\Psi_{2*}\gamma(x,x',y) &= \int d(\psi_y(x),\psi_y(x'))^q \dd\gamma(x,x',y) \\
  &\le \rho^q \int d(x,x')^q \dd\gamma(x,x',y)
\end{align*}
and taking an infimum, we see that
\[\wass^\nu_q(\Psi_*\sigma_0,\Psi_*\sigma_1) \le \rho^{\min(1,q)}\wass^\nu_q(\sigma_0,\sigma_1),\]
in particular $\Psi_*$ induces a contraction on the complete metric space $(\proba^\nu_q,\wass^\nu_q)$. Therefore, there exists a unique $\Psi$-invariant lift $\hat\nu$ of $\nu$ having finite $q$-th moment. By considering different $q$, we already see that the measure $\hat\nu$ has finite moments of all orders but, since $B\times Y$ is absorbing, any $\Psi$-invariant measure is concentrated on $B\times Y$. This proves that $\hat\nu$ is the unique $\Psi$-invariant lift of $\nu$ on the whole of $\proba^\nu$, and that its first marginal $\mu$ is supported on a bounded set. Explicitly, by letting $\varepsilon$ above go to $0$, we obtain that $\mu$ is concentrated on $B(x_0,A/(1-\rho))$.

Consider now the stochastic process $(\rv{X}_k)_{k\in\mathbb{N}}$. Let $\sigma_0:= \tilde \mu_0 \otimes \nu$ be the law of $(\rv{X}_0,\rv{Y})$; then the law of $(\rv{X}_k,S^k(\rv{Y})) = \Psi^k(\rv{X}_0,\rv{Y})$ is $\sigma_k =\Psi^k_*(\sigma_0)$, by definition has first marginal $\tilde\mu_k$, and by invariance has second marginal $\nu$. Since $\Psi_*$ is a contraction in $\proba^\nu_q\ni\sigma_0$, we obtain that
\begin{equation}\wass^\nu_q(\sigma_k,\hat\nu) \le \rho^{k \min(1,q)} \wass^\nu_q(\sigma_0,\hat\nu).
\label{e:wassNu}
\end{equation}
On the first hand, using the transport plan obtained by projecting an optimal $\gamma\in \Gamma^\nu(\sigma_k,\hat\nu)$ on the first two variables, we get $\wass_q(\tilde\mu_k,\mu) \le \wass^\nu_q(\sigma_k,\hat\nu)$. On the other hand,
\begin{align*}
\wass^\nu_q(\sigma_0,\hat\nu) &\le \wass^\nu_q(\tilde\mu_0\otimes \nu, \delta_{x_0}\otimes \nu) + \wass^\nu_q(\delta_{x_0}\otimes \nu,\hat\nu) \\
  &\le m_{x_0}^q(\tilde\mu_0)^{\min(1,\frac1q)} + \big( A/(1-\rho) \big)^{\min(1,q)}
\end{align*}
since $\hat\nu$ is concentrated on $B(x_0,A/(1-\rho))\times Y$.
Together with \eqref{e:wassNu}, this concludes the proof of Theorem \ref{th:skew}.

\bibliographystyle{amsalpha}
\bibliography{stationary-IFS}

\newcommand{\etalchar}[1]{$^{#1}$}
\providecommand{\bysame}{\leavevmode\hbox to3em{\hrulefill}\thinspace}
\providecommand{\MR}{\relax\ifhmode\unskip\space\fi MR }
\providecommand{\MRhref}[2]{%
  \href{http://www.ams.org/mathscinet-getitem?mr=#1}{#2}
}
\providecommand{\href}[2]{#2}
\begin{thebibliography}{BDM{\etalchar{+}}88}

\bibitem[AGS08]{ambrosio2008gradient}
Luigi Ambrosio, Nicola Gigli, and Giuseppe Savar{\'e}, \emph{Gradient flows: in
  metric spaces and in the space of probability measures}, Springer, 2008.

\bibitem[AH16]{anckar2016fine}
Andreas Anckar and G\"{o}ran H\"{o}gn\"{a}s, \emph{The fine structure of the
  stationary distribution for a simple {M}arkov process}, Probability on
  algebraic and geometric structures, Contemp. Math., vol. 668, Amer. Math.
  Soc., Providence, RI, 2016, pp.~1--12. \MR{3536688}

\bibitem[BDEG88]{barnsley1988placeDependent}
M.~F. Barnsley, S.~G. Demko, J.~H. Elton, and J.~S. Geronimo, \emph{Invariant
  measures for {M}arkov processes arising from iterated function systems with
  place-dependent probabilities}, Ann. Inst. H. Poincar\'{e} Probab. Statist.
  \textbf{24} (1988), no.~3, 367--394. \MR{971099}

\bibitem[BDM{\etalchar{+}}88]{barnsley1988book}
Michael~F. Barnsley, Robert~L. Devaney, Benoit~B. Mandelbrot, Heinz-Otto
  Peitgen, Dietmar Saupe, and Richard~F. Voss, \emph{The science of fractal
  images}, Springer-Verlag, New York, 1988, With contributions by Yuval Fisher
  and Michael McGuire. \MR{952853}

\bibitem[BE88]{barnsley1988newclass}
Michael~F. Barnsley and John~H. Elton, \emph{A new class of {M}arkov processes
  for image encoding}, Adv. in Appl. Probab. \textbf{20} (1988), no.~1, 14--32.
  \MR{932532}

\bibitem[BMS06]{bergelson2006affine}
Vitaly Bergelson, Micha{\L} Misiurewicz, and Samuel Senti, \emph{Affine actions
  of a free semigroup on the real line}, Ergodic Theory Dynam. Systems
  \textbf{26} (2006), no.~5, 1285--1305.

\bibitem[BS12]{baladi-smania2012linear}
Viviane Baladi and Daniel Smania, \emph{Linear response for smooth deformations
  of generic nonuniformly hyperbolic unimodal maps}, Ann. Sci. \'Ec. Norm.
  Sup\'er. (4) \textbf{45} (2012), no.~6, 861--926 (2013). \MR{3075107}

\bibitem[BV11]{barnsley2011chaos}
Michael~F. Barnsley and Andrew Vince, \emph{The chaos game on a general
  iterated function system}, Ergodic Theory Dynam. Systems \textbf{31} (2011),
  no.~4, 1073--1079. \MR{2818686}

\bibitem[DF99]{diaconis1999iterated}
Persi Diaconis and David Freedman, \emph{Iterated random functions}, SIAM
  review \textbf{41} (1999), no.~1, 45--76.

\bibitem[DM19]{dedecker2019moments}
{J}{\'e}r{\^o}me {D}edecker and {F}lorence {M}erlev{\`e}de, \emph{Behavior of
  the empirical {W}asserstein distance in $\mathbb{R}^d$ under moment
  conditions}, Electron. J. Probab. \textbf{24} (2019).

\bibitem[Elt87]{elton1987ergodic}
John~H. Elton, \emph{An ergodic theorem for iterated maps}, Ergodic Theory
  Dynam. Systems \textbf{7} (1987), no.~4, 481--488. \MR{922361}

\bibitem[Elt90]{elton1990ergodic}
\bysame, \emph{A multiplicative ergodic theorem for {L}ipschitz maps},
  Stochastic Process. Appl. \textbf{34} (1990), no.~1, 39--47. \MR{1039561}

\bibitem[FG15]{fournier2015rate}
Nicolas Fournier and Arnaud Guillin, \emph{On the rate of convergence in
  {W}asserstein distance of the empirical measure}, Probab. Theory Related
  Fields \textbf{162} (2015), no.~3-4, 707--738.

\bibitem[FM98]{forte1998ergodic}
B.~Forte and F.~Mendivil, \emph{A classical ergodic property for {IFS}: a
  simple proof}, Ergodic Theory Dynam. Systems \textbf{18} (1998), no.~3,
  609--611. \MR{1631724}

\bibitem[Fra15]{fraser2015moments}
Jonathan~M. Fraser, \emph{First and second moments for self-similar couplings
  and {W}asserstein distances}, Math. Nachr. \textbf{288} (2015), no.~17-18,
  2028--2041. \MR{3434297}

\bibitem[GKLM18]{GKLM2018calculus}
Paolo Giulietti, Beno\^{\i}t Kloeckner, Artur~O. Lopes, and Diego Marcon,
  \emph{The calculus of thermodynamical formalism}, J. Eur. Math. Soc. (JEMS)
  \textbf{20} (2018), no.~10, 2357--2412. \MR{3852182}

\bibitem[GMN16]{galatolo2016approximation}
Stefano Galatolo, Maurizio Monge, and Isaia Nisoli, \emph{Rigorous
  approximation of stationary measures and convergence to equilibrium for
  iterated function systems}, J. Phys. A \textbf{49} (2016), no.~27, 274001,
  22. \MR{3512100}

\bibitem[Gol91]{goldie1991implicit}
Charles~M Goldie, \emph{Implicit renewal theory and tails of solutions of
  random equations}, Ann. Appl. Probab. \textbf{1} (1991), no.~1, 126--166.

\bibitem[Hut81]{hutchinson1981fractals}
John~E. Hutchinson, \emph{Fractals and self-similarity}, Indiana Univ. Math. J.
  \textbf{30} (1981), no.~5, 713--747. \MR{625600}

\bibitem[Ios09]{iosifescu2009criticalSurvey}
Marius Iosifescu, \emph{Iterated function systems. {A} critical survey}, Math.
  Rep. (Bucur.) \textbf{11(61)} (2009), no.~3, 181--229. \MR{2551080}

\bibitem[JO10]{joulin-ollivier2010curvature}
Ald\'eric Joulin and Yann Ollivier, \emph{Curvature, concentration and error
  estimates for {M}arkov chain {M}onte {C}arlo}, Ann. Probab. \textbf{38}
  (2010), no.~6, 2418--2442. \MR{2683634}

\bibitem[Kes73]{kesten1973random}
Harry Kesten, \emph{Random difference equations and renewal theory for products
  of random matrices}, Acta Math. \textbf{131} (1973), no.~1, 207--248.

\bibitem[Kev16]{kevei2016noteKGG}
P\'{e}ter Kevei, \emph{A note on the {K}esten-{G}rincevi\v{c}ius-{G}oldie
  theorem}, Electron. Commun. Probab. \textbf{21} (2016), Paper No. 51, 12.
  \MR{3533283}

\bibitem[Klo18a]{kloeckner2018empirical}
Beno\^{\i}t~R. Kloeckner, \emph{Empirical measures: regularity is a
  counter-curse to dimensionality}, arXiv:1802.04038, 2018.

\bibitem[Klo18b]{kloeckner2018shrinking}
\bysame, \emph{Extensions with shrinking fibers}, arXiv:1812.08437, 2018.

\bibitem[MS10]{madras2010quantitative}
Neal Madras and Deniz Sezer, \emph{Quantitative bounds for {M}arkov chain
  convergence: {W}asserstein and total variation distances}, Bernoulli
  \textbf{16} (2010), no.~3, 882--908. \MR{2730652}

\bibitem[NSB02]{nicol2002fine}
Matthew Nicol, Nikita Sidorov, and David Broomhead, \emph{On the fine structure
  of stationary measures in systems which contract-on-average}, J. Theoret.
  Probab. \textbf{15} (2002), no.~3, 715--730.

\bibitem[Oll09]{ollivier2009ricci}
Yann Ollivier, \emph{Ricci curvature of {M}arkov chains on metric spaces}, J.
  Funct. Anal. \textbf{256} (2009), no.~3, 810--864. \MR{2484937}

\bibitem[Pei93]{peigne1993IFSspectral}
Marc Peign\'{e}, \emph{Iterated function systems and spectral decomposition of
  the associated {M}arkov operator}, Fascicule de probabilit\'{e}s, Publ. Inst.
  Rech. Math. Rennes, vol. 1993, Univ. Rennes I, Rennes, 1993, p.~28.
  \MR{1347702}

\bibitem[Pol01]{pollicott2001transfer}
M.~Pollicott, \emph{Contraction in mean and transfer operators}, Dyn. Syst.
  \textbf{16} (2001), no.~1, 97--106. \MR{1835908}

\bibitem[Rue98]{ruelle1998linear}
David Ruelle, \emph{General linear response formula in statistical mechanics,
  and the fluctuation-dissipation theorem far from equilibrium}, Phys. Lett. A
  \textbf{245} (1998), no.~3-4, 220--224. \MR{1642617}

\bibitem[Rue09]{ruelle2009review}
\bysame, \emph{A review of linear response theory for general differentiable
  dynamical systems}, Nonlinearity \textbf{22} (2009), no.~4, 855--870.
  \MR{2486360}

\bibitem[Sol95]{solomyak1995erdos}
Boris Solomyak, \emph{On the random series {$\sum\pm\lambda^n$} (an
  {E}rd{\H{o}}s problem)}, Ann. of Math. (2) \textbf{142} (1995), no.~3,
  611--625. \MR{1356783}

\bibitem[SS98]{silvestrov1998ergodic}
Dmitrii~S. Silvestrov and \"{O}rjan Stenflo, \emph{Ergodic theorems for
  iterated function systems controlled by regenerative sequences}, J. Theoret.
  Probab. \textbf{11} (1998), no.~3, 589--608. \MR{1633370}

\bibitem[Ste99]{steinsaltz1999locally}
David Steinsaltz, \emph{Locally contractive iterated function systems}, Ann.
  Probab. \textbf{27} (1999), no.~4, 1952--1979. \MR{1742896}

\bibitem[SW13]{santos2013limitLaws}
Sara~I. Santos and Charles Walkden, \emph{Distributional and local limit laws
  for a class of iterated maps that contract on average}, Stoch. Dyn.
  \textbf{13} (2013), no.~2, 1250019, 28. \MR{3039419}

\bibitem[Sza03]{szarek2003polish}
Tomasz Szarek, \emph{Invariant measures for nonexpensive {M}arkov operators on
  {P}olish spaces}, Dissertationes Math. (Rozprawy Mat.) \textbf{415} (2003),
  62, Dissertation, Polish Academy of Science, Warsaw, 2003. \MR{1997024}

\bibitem[Var18]{varju2018recent}
P\'{e}ter~P. Varj\'{u}, \emph{Recent progress on {B}ernoulli convolutions},
  European {C}ongress of {M}athematics, Eur. Math. Soc., Z\"{u}rich, 2018,
  pp.~847--867. \MR{3890454}

\bibitem[Vil09]{Villani2009OldNew}
C{\'e}dric Villani, \emph{Optimal transport}, Grundlehren der Mathematischen
  Wissenschaften [Fundamental Principles of Mathematical Sciences], vol. 338,
  Springer-Verlag, Berlin, 2009, Old and new. \MR{MR2459454}

\bibitem[Wal07]{walkden2007invariancePrinciple}
C.~P. Walkden, \emph{Invariance principles for iterated maps that contract on
  average}, Trans. Amer. Math. Soc. \textbf{359} (2007), no.~3, 1081--1097.
  \MR{2262842}

\end{thebibliography}
\addcontentsline{toc}{section}{References}

\end{document}